\documentclass[]{elsarticle}
\usepackage{amssymb, amsthm, enumerate, amsmath, mathrsfs}
\usepackage{tikz}
\usetikzlibrary{calc,patterns,through}

\theoremstyle{plain}
\newtheorem{thm}{Theorem}[section]
\newtheorem{lemma}[thm]{Lemma}
\newtheorem{cor}[thm]{Corollary}
\newtheorem{claim}{}[thm]
\newenvironment{subproof}{\begin{proof}[Subproof.]}{\end{proof}}

\DeclareMathOperator{\fcl}{fcl}

\newcommand{\del}{\hspace{-0.5pt}\backslash}
\newcommand{\tangle}{{\mathcal T}}

\begin{document}
 \title{Tangles, trees, and flowers}

\author[vuw]{Ben Clark}
\ead{benjamin.clark@msor.vuw.ac.nz}

\author[vuw]{Geoff Whittle \fnref{gw}}
\ead{geoff.whittle@vuw.ac.nz}

\fntext[gw]{Geoff Whittle was supported by the Marsden Fund of New Zealand.}

\address[vuw]{School of Mathematics, Statistics, and Operations Research, Victoria University of Wellington, P.O. Box 600, Wellington, New Zealand.}


\begin{abstract}
A tangle of order $k$ in a matroid or graph may be thought of as a ``$k$-connected component''. For a tangle of order $k$ in a matroid or graph that satisfies a certain robustness condition, we describe a tree decomposition of the matroid or graph that displays, up to a certain natural equivalence, all of the $k$-separations of the matroid or graph that are non-trivial with respect to the tangle.
\end{abstract}

\maketitle

\section{Introduction}
\label{intro}

The structure of the $3$-separations of $3$-connected matroids is described by Oxley, Semple, and Whittle \cite{oxley2004structure,oxley2007structure}. In particular, they show that every $3$-connected matroid $M$ with at least nine elements has a tree decomposition that displays, up to a certain natural equivalence, all of the ``non-trivial'' $3$-separations of $M$. This result was generalised by Aikin and Oxley \cite{aikin2011structure} to describe the structure of the $4$-separations of $4$-connected matroids. They show that every $4$-connected matroid $M$ with at least 17 elements has a tree decomposition that displays, up to a natural equivalence, all of the ``non-trivial'' $4$-separations of $M$. 

It is natural to believe
that analogous structural results will hold for more highly connected matroids. 
Unfortunately strict $k$-connectivity becomes an increasingly artificial requirement 
even for modest values of $k$; for example, projective geometries are not even
4-connected. More realistically, a matroid may have identifiable regions of 
high connectivity. Such regions are captured by the notion of a ``tangle'' in a matroid.
It is also not difficult to see that the existence of the tree decomposition for
3- and 4-connected matroids relies primarily on the fact that the connectivity function of 
a matroid is symmetric and submodular, in other words, it is a ``connectivity system''.
Tangles were introduced for graphs by Robertson and Seymour \cite{robertson1991graph}, 
and were extended to connectivity systems by Geelen, Gerards, Robertson, and 
Whittle \cite{geelen2006obstructions}. 

In this paper we prove theorems about the structure of $k$-separations in 
tangles of order $k$ in connectivity systems.
Loosely speaking, the main result of this paper says that every ``robust'' 
tangle of order $k$ in a matroid $M$ has a tree decomposition that displays, 
up to a certain natural equivalence, all of the $k$-separations of $M$ that 
are ``non-trivial'' with respect to the tangle.
In doing so we are operating at a high level of generality. It is, of course,
the consequences for more particular structures that are of primary interest. 
Let $M$ be a matroid with rank function $r$. Recall that the {\em connectivity
function} $\lambda_M$ of a matroid $M$ on $E$
is defined by $\lambda_M(X)=r(X)+r(E-X)-r(M)+1$ for 
all subsets $X$ of $E$. (Note that we have included the notorious $+1$ in the
definition of $\lambda_M$ in this paper.) The pair $(E,\lambda_M)$ is a connectivity system and
the $k$-separations of a matroid are precisely the $k$-separations of its 
associated connectivity system, so we obtain results for $k$-separations in
tangles of order $k$ in matroids as immediate consequences of our theorems on
connectivity systems. We also obtain consequences for graphs as follows.
Let $G$ be a graph with edge set $E$. For $X\subseteq E$, we let $\lambda_G(X)$ 
denote the number of vertices of $G$ that are incident with both an edge of $X$ 
and an edge of $E-X$. It is not difficult to prove that $(E,\lambda_G)$ is a 
connectivity system. Via this connectivity system
we also obtain consequences for $k$-separations
in tangles of order $k$ in graphs. In this paper we do not further consider applications to 
graphs; nonetheless, we do believe that such applications are potentially of 
some interest.

A matroid may have a unique tangle of a given order. For example, 3- and 4-connected
matroids have unique tangles of order 3 and 4 respectively. Applying the results of this
paper to these tangles gives the above-mentioned tree structures for 
3- and 4-connected matroids. 
We make this connection explicit, but rather than restrict our attention to strict
$k$-connectivity, we focus on a somewhat more general notion of connectivity. Let
$k\geq 2$ be an integer. 
A matroid $M$ is {\em loosely vertically $k$-connected} if, for each $(k-1)$-separation 
$(X,Y)$ of $M$, either $r(X)\leq k-2$ or $r(Y)\leq k-2$. 
A loosely vertically $k$-connected matroid with rank at least $3k-4$ has a tangle of
order $k$. Moreover, this tangle is unique. 
Via an appropriate interpretation we obtain, in each section, results for 
loosely vertically $k$-connected matroids. In particular we prove that, up to a 
natural notion of equivalence, all of the ``interesting'' $k$-separations in 
a loosely vertically $k$-connected matroids can be displayed in a tree-like
way. Some readers may be more interested in these consequences than in the 
more general results
for connectivity systems. To this end, we have placed this matter in subsections at the
end of each section. 

A matroid $M$ is {\em vertically $k$-connected} if, for each $l$-separation $(X,Y)$,
where $l<k$, either $r(X)<l$ or $r(Y)<l$. This is an important connectivity notion
that generalises $k$-connectivity for graphs, see, for example Oxley \cite[Chapter~8.6]{oxley2011matroid}.
We have no wish to supplant this terminology, but, to avoid cumbersome statements,
for the remainder of this paper ``vertically $k$-connected'' will mean
``loosely vertically $k$-connected''. 

Another tree that can be associated with a connectivity system---and hence a matroid or
graph---is the so-called ``tree of tangles''; see for example 
\cite{geelen2009tangles,robertson1991graph}. This is a tree that displays in a certain way all
of the maximal tangles. Ideally one would like to be able associate a tree with a 
connectivity system that displays all of the maximal tangles and representatives of all
of the interesting separations within tangles. Unfortunately this does not seem possible
with the separations that we focus on in this paper. However, it might be possible to do
this if attention is restricted to certain subsets of separations. With this end in
mind we define what it means for a collection of separations to be ``tree compatible''
in Section~3. A tree-compatible collection of separations can be displayed in a tree-like
way and it may be that we can associate a tree with a connectivity system that simultaneously 
displays all of the maximal tangles and certain interesting collections of
tree-compatible separations within the maximal tangles. We leave this as a question for
future research.

More specifically the paper is organised as follows. Section \ref{basics} contains 
basic definitions and results about connectivity systems and tangles. In Section \ref{ksep} 
we define a notion of equivalence on the $k$-separations of a connectivity system with 
respect to a fixed tangle of order $k$, and we also define what it means for a 
$k$-separation of a connectivity system to be sequential with respect to a fixed 
tangle of order $k$. In section \ref{flowtang} we introduce $k$-flowers, which 
enable us to display a collection of crossing $k$-separations. In Section \ref{conf} 
we study how the non-sequential $k$-separations can interact with $k$-flowers, and 
prove Theorem \ref{TMkflowerconf}, which says that if we introduce a notion of robustness 
for tangles, then all of the non-sequential $k$-separations interact with a ``maximal'' 
$k$-flower in a coherent way. Section \ref{ktree} contains definitions and results 
about tree decompositions. Finally, in Section \ref{mainone} we prove our main theorem.

\section{Connectivity systems and tangles}
\label{basics}

Let $E$ be a finite set, and let $\lambda$ be an integer-valued function on the subsets 
of $E$. We call $\lambda$ \textit{symmetric} if $\lambda(X)=\lambda(E-X)$ for all 
$X\subseteq E$. We call $\lambda$ \textit{submodular} if 
$\lambda(X)+\lambda(Y)\geq \lambda(X\cup Y)+\lambda(X\cap Y)$ for all $X,Y\subseteq E$. 
A \textit{connectivity function} on $E$ is an integer-valued function on the subsets of 
$E$ that is both symmetric and submodular. A \textit{connectivity system} is an ordered 
pair $(E,\lambda)$ consisting of a finite set $E$ and a connectivity function $\lambda$ on $E$. 

Let $(E,\lambda)$ be a connectivity system, and let $k$ be a positive integer. 
A partition $(X,E-X)$ of $E$ is called a \textit{$k$-separation} of $\lambda$ if 
$\lambda(X)\leq k$. A subset $X$ of $E$ is said to be \textit{$k$-separating} in $\lambda$ 
if $\lambda(X)\leq k$. When the connectivity function $\lambda$ is clear from the context 
we shall often 
abbreviate ``$k$-separation of $\lambda$'' and ``$k$-separating set in $\lambda$'' to 
``$k$-separation'' and ``$k$-separating set'' respectively. A $k$-separating set $X$, or 
$k$-separation $(X,E-X)$, is \textit{exact} if $\lambda(X)=k$. Note that we consider 
a $k$-separation $(X,E-X)$ to be an unordered partition of $E$, and we make no 
assumptions on the number of elements in the sets $X$ and $E-X$.

We will make use of the following elementary properties of connectivity functions.

\begin{lemma}
\label{elemconn}\cite[Lemma 2.3.]{geelen2006obstructions}
If $\lambda$ is a connectivity function on $E$, then, for all $X,Y\subseteq E$, we have:\begin{enumerate}
 \item[(i)] $\lambda(X)\geq \lambda(\emptyset)$.
 \item[(ii)] $\lambda(X)+\lambda(Y)\geq \lambda(X-Y)+\lambda(Y-X)$.
\end{enumerate}
\end{lemma}

The submodularity of a connectivity function $\lambda$ is frequently used in the following form, and we write \textit{by uncrossing $X$ and $Y$} to mean ``by an application of Lemma \ref{uncrossing2}''. 

\begin{lemma}
\label{uncrossing2}
Let $X$ and $Y$ be $k$-separating subsets of $E$.

(i) If $\lambda(X\cap Y)\geq k$, then $X\cup Y$ is $k$-separating.

(ii) If $\lambda(E-(X\cup Y))\geq k$, then $X\cap Y$ is $k$-separating.
\end{lemma}

Tangles were introduced for graphs by Robertson and Seymour \cite{robertson1991graph}, and were extended to connectivity systems by Geelen, Gerards, Robertson, and Whittle \cite{geelen2006obstructions}. For a positive integer $k$, a \textit{tangle of order $k$} in a connectivity system $(E,\lambda)$ is a collection $\mathcal{T}$ of subsets of $E$ such that the following properties hold:
\begin{enumerate}
 \item[(T1)] $\lambda(A)<k$ for all $A\in \mathcal{T}$.
 \item[(T2)] If $(A,E-A)$ is a $(k-1)$-separation, then $\mathcal{T}$ contains $A$ or $E-A$.
 \item[(T3)] If $A,B,C\in \mathcal{T}$, then $A\cup B\cup C\not= E$.
 \item[(T4)] $E-\{e\}\notin \mathcal{T}$ for each $e\in E$.
\end{enumerate}

Let $\mathcal{T}$ be a tangle of order $k$ in a connectivity system $(E,\lambda)$. 
A subset $X$ of $E$ is \textit{$\mathcal{T}$-strong} if it is not contained in a 
member of $\mathcal{T}$; otherwise $X$ is \textit{$\mathcal{T}$-weak}. 
It is easy to see that supersets of $\mathcal{T}
$-strong sets are $\mathcal{T}$-strong, and that subsets of $\mathcal{T}$-weak sets 
are $\mathcal{T}$-weak. A partition $(X_1,\ldots, X_n)$ of $E$ is \textit{$\mathcal{T}$-strong} 
if $X_i$ is a $\mathcal{T}$-strong set for all $i\in [n]$; otherwise $(X_1,\ldots, X_n)$ 
is \textit{$\mathcal{T}$-weak}. In particular, a $k$-separation $(X,E-X)$ of $\lambda$ 
is \textit{$\mathcal{T}$-strong} if both $X$ and $E-X$ are $\mathcal{T}$-strong sets; 
otherwise $(X,E-X)$ is \textit{$\mathcal{T}$-weak}. A number of uncrossing arguments 
make use of the fact that, if a partition $(X,E-X)$ of $E$ is $\mathcal{T}$-strong, 
then neither $X$ nor $E-X$ is a member of $\mathcal{T}$, so $\lambda(X)\geq k$ by (T2).

In any unexplained context, if we use the phrase ``$\mathcal{T}$-strong 
$k$-separating set'' or ``$\mathcal{T}$-strong $k$-separation'' without 
mention of the order of the tangle $\mathcal{T}$, then it will be implicit 
that $\mathcal{T}$ has order $k$.

\subsection*{Tangles in vertically $k$-connected matroids} 
Recall that, for $k\geq 2$,  a matroid $M$ is {\em loosely vertically $k$-connected} if,
for every $(k-1)$-separation $(A,B)$ of $M$, either $r(A)\leq k-2$ or $r(B)\leq k-2$
and that when we say that a matroid is ``vertically $k$-connected'' we will mean that it is
``loosely vertically $k$-connected.''
Degeneracies can arise because of low rank, but once past these, a vertically
$k$-connected matroid has a unique tangle.

\begin{lemma}
\label{one-tangle}
Let $M$ be a vertically $k$-connected matroid on $E$, where $k\geq 2$
and $r(M)\geq \max\{3k-5,2\}$. Then $M$ has a unique tangle $\mathcal T$ of 
order $k$. Moreover a subset $A$ belongs to $\mathcal T$ if and only if 
$r(A)\leq k-2$.
\end{lemma}

\begin{proof}
Let ${\mathcal T}=\{X\subseteq E\ |\ r(X)\leq k-2\}$. We first show that
$\mathcal T$ is a tangle of order $k$. Property (T1) follows immediately
while (T2) follows from the definition of vertical $k$-connectivity.
Say $A,B,C\in\mathcal T$. Then $r(A\cup B\cup C)\leq 3(k-2)<r(M)$,
so $A\cup B\cup C\neq E$. Thus (T3) holds.

Consider (T4). Say $E-\{e\}\in\mathcal T$ for some $e\in E$. Then 
$r(E-\{e\})\leq k-2$. If $k\geq 3$, this immediately contradicts the fact that
$r(M)\geq 3k-5$. In the case that $k=2$, we see that $r(E-\{e\})=0$, so that
$r(M)\leq 1$, contradicting the fact that $r(M)\geq 2$. We conclude that (T4)
also holds so that $\mathcal T$ is indeed a tangle in $M$ of order $k$.

It remains to show that $\mathcal T$ is unique. Let ${\mathcal T}'$ be a
tangle of order $k$ in $M$. Say $(A,B)$ is a $(k-1)$-separation.
By the definition of vertical $k$-connectivity we may assume that
$r(A)\leq k-2$. Assume that $B\in{\mathcal T}'$. By (T3) $A\neq \emptyset$. 
Choose $a\in A$. Then $r(A-\{a\})\leq k-2$ and $r(\{a\})\leq k-2$.
By (T4), $\{a\}\in {\mathcal T}'$. If $A-\{a\}\in{\mathcal T}'$,
then we contradict (T3). Thus $B\cup\{a\}\in{\mathcal T}'$. Iterating this 
procedure leads to a contradiction of (T4). Thus $B\notin{\mathcal T}'$, so that
$A\in{\mathcal T}'$ and we conclude that ${\mathcal T}'=\mathcal T$.
\end{proof}

Interpreting Lemma~\ref{one-tangle} for $k$-connected matroids, we obtain

\begin{cor}
\label{one-tangle-cor}
Let $M$ be a $k$-connected matroid on $E$, where $k\geq 2$
and $|E|\geq \max\{3k-5,2\}$. Then $M$ has a unique tangle $\mathcal T$ of 
order $k$. Moreover a subset $A$ belongs to $\mathcal T$ if and only if 
$|A|\leq k-2$.
\end{cor}

For a vertically $k$-connected matroid $M$ with $r(M)\geq \max\{3k-5,2\}$,
we denote the unique tangle of order $k$ in $M$ by ${\mathcal T}_k$.
Thus a subset $A\subseteq E(M)$ is ${\mathcal T}_k$-weak if and only if
$r(A)\leq k-2$ and the partition $(A_1,A_2,\ldots,A_n)$ is 
$\tangle_k$-strong if and only if $r(A_i)\geq k-1$ for all $i\in\{1,2,\ldots,n\}$.
In particular, if $M$ is $k$-connected, then $A\subseteq E(M)$ is $\tangle_k$-weak 
if and only if
$|A|\leq k-2$ and the partition $(A_1,A_2,\ldots,A_n)$ is 
$\tangle_k$-strong if and only if $|A|\geq k-1$ for all $i\in\{1,2,\ldots,n\}$.

\section{Sequential and equivalent $k$-separations}
\label{ksep}

In this section we fix a tangle $\mathcal{T}$ of order $k$ in a connectivity 
system $(E,\lambda)$, and we focus on the $k$-separations of $\lambda$ 
that are $\mathcal{T}$-strong. We then define a natural notion of equivalence 
on the $\mathcal{T}$-strong $k$-separations of $\lambda$, and we define what it 
means for a $\mathcal{T}$-strong $k$-separation to be sequential with respect to $\mathcal{T}$.

Let $\mathcal{T}$ be a tangle of order $k$ 
in a connectivity system $(E,\lambda)$. A $\mathcal{T}$-strong $k$-separating set $X$ 
is \textit{fully closed with respect to 
$\mathcal{T}$} if $\lambda(X\cup Y)>k$ is not $k$-separating for every non-empty 
$\mathcal{T}$-weak set 
$Y\subseteq E-X$. In particular, we observe that if a $k$-separating set $X$ 
is a proper subset of $E$ that is fully closed with respect to $\mathcal{T}$, 
then $E-X$ is a $\mathcal{T}$-strong set because $X\cup (E-X)$ is $k$-separating 
by Lemma \ref{elemconn} (i). We abbreviate ``fully closed with respect to $\mathcal{T}$'' 
to ``fully closed'' when the tangle $\mathcal{T}$ is clear from the context.

We shall show that every $\mathcal{T}$-strong $k$-separating set is contained in a unique minimal $k$-separating set that is fully closed with respect to $\mathcal{T}$. The following lemma is the main step towards this.

\begin{lemma}
\label{intsectfcl}
Let $\mathcal{T}$ be a tangle of order $k$ in a connectivity system $(E,\lambda)$, and let $X$ be a $\mathcal{T}$-strong $k$-separating set in $\lambda$. If $X_1$ and $X_2$ are fully-closed $k$-separating sets that contain $X$, then there is a fully-closed $k$-separating set $Y$ such that $X\subseteq Y\subseteq X_1\cap X_2$.
\end{lemma}

\begin{proof}
Let $Y$ be maximal with respect to the properties that $X\subseteq Y\subseteq X_1\cap X_2$ and $\lambda(Y)\leq k$. Assume towards a contradiction that $Y$ is not fully closed. Then there exists a non-empty $\mathcal{T}$-weak set $Z$ contained in $E-Y$ such that $Y\cup Z$ is $k$-separating. If $Z\subseteq X_1\cap X_2$, then $X\subseteq Y\cup Z\subseteq X_1\cap X_2$ and $Y\cup Z$ properly contains $Y$; a contradiction of the maximality of $Y$. Therefore, up to switching $X_1$ and $X_2$, we may assume that $Z$ meets $E-X_1$. Now, the partition $(Y\cup (Z\cap X_1), E-(Y\cup (Z\cap X_1)))$ is $\mathcal{T}$-strong because $Y\cup (Z\cap X_1)$ contains the $\mathcal{T}$-strong set $X$ while $E-(Y\cup (Z\cap X_1))$ contains the $\mathcal{T}$-strong set $E-X_1$, so $\lambda(Y\cup (Z\cap X_1))\geq k$ by (T2). But $Y\cup (Z\cap X_1)$ is the intersection of the $k$-separating sets $Y\cup Z$ and $X_1$, so their union $X_1\cup Z$ is $k$-separating by uncrossing; a contradiction because $X_1$ is fully closed.  
\end{proof}

We omit the proof of the next result, which follows from Lemma \ref{intsectfcl} and a straightforward induction.

\begin{cor}
 \label{fcldefn}
Let $\mathcal{T}$ be a tangle of order $k$ in a connectivity system $(E,\lambda)$. Let $X$ be a $\mathcal{T}$-strong $k$-separating set, and let $\mathcal{F}$ be the set of fully-closed $k$-separating sets that contain $X$. Then $\bigcap \mathcal{F}$ is a fully-closed $k$-separating set that contains $X$.
\end{cor}

Let $\mathcal{T}$ be a tangle of order $k$ in a connectivity system $(E,\lambda)$, and let $X$ be a $\mathcal{T}$-strong $k$-separating set. Then the intersection of all fully-closed $k$-separating sets that contain $X$, which we denote by $\fcl_{\mathcal{T}}(X)$, is called the \textit{full closure of $X$ with respect to $\mathcal{T}$}. By Corollary \ref{fcldefn}, we see that $\fcl_{\mathcal{T}}(X)$ is minimal with respect to being a fully-closed $k$-separating set that contains $X$. We abbreviate ``full closure with respect to $\mathcal{T}$'' to ``full closure'' when the tangle $\mathcal{T}$ is clear from the context.

Let $\mathcal{T}$ be a tangle of order $k$ in a connectivity system $(E,\lambda)$. 
We omit the routine proof of the next lemma, which shows that the full closure 
with respect to $\mathcal{T}$ is a closure operator on the set of $\mathcal{T}$-strong 
$k$-separating sets in $\lambda$.

\begin{lemma}
\label{fcloperator}
Let $\mathcal{T}$ be a tangle of order $k$ in a connectivity system $(E,\lambda)$. Let $X$ and $Y$ be $\mathcal{T}$-strong $k$-separating sets. Then the following hold:  
\begin{enumerate}
 \item[(i)] $X\subseteq \fcl_{\mathcal{T}}(X)$.
 \item[(ii)] If $X\subseteq Y$, then $\fcl_{\mathcal{T}}(X)\subseteq \fcl_{\mathcal{T}}(Y)$.
 \item[(iii)] $\fcl_{\mathcal{T}}(\fcl_{\mathcal{T}}(X))=\fcl_{\mathcal{T}}(X)$. 
\end{enumerate}
\end{lemma}

Let $\mathcal{T}$ be a tangle of order $k$ in a connectivity system $(E,\lambda)$. 
Let $X$ and $Y$ be $\mathcal{T}$-strong $k$-separating sets in $\lambda$. 
We say that $X$ is $\mathcal{T}$-\textit{equivalent} to $Y$ if 
$\fcl_{\mathcal{T}}(X)= \fcl_{\mathcal{T}}(Y)$. It is easy to see that 
$\mathcal{T}$-equivalence is an equivalence relation on the set of 
$\mathcal{T}$-strong $k$-separating sets in $\lambda$. In what follows, we may suppress the tangle $\mathcal{T}$ and say that $X$ is ``equivalent'' to $Y$ when the tangle $\mathcal{T}$ is clear from the context.

Let $X$ be a $\mathcal{T}$-strong $k$-separating set in $\lambda$. A \textit{partial $k$-sequence for $X$} \index{partial $k$-sequence} is a sequence $(X_i)_{i=1}^m$ of pairwise disjoint, non-empty $\mathcal{T}$-weak subsets of $E-X$ such that $X\cup (\bigcup_{i=1}^j X_i)$ is $k$-separating for all $j\in [m]$.    

\begin{lemma}
\label{fclcontain}
Let $\mathcal{T}$ be a tangle of order $k$ in a connectivity system $(E,\lambda)$, and let $X$ be a $\mathcal{T}$-strong $k$-separating set in $\lambda$. If $(X_i)_{i=1}^m$ is a partial $k$-sequence for $X$, then  $X\cup (\bigcup_{i=1}^m X_i)\subseteq \fcl_{\mathcal{T}}(X)$.
\end{lemma}

\begin{proof}
Suppose that $(X_i)_{i=1}^m$ is a partial $k$-sequence for $X$ such that $X\cup (\bigcup_{i=1}^m X_i)$ is not contained in $\fcl_{\mathcal{T}}(X)$. Then $\fcl_{\mathcal{T}}(X)\neq E$, so $E-\fcl_{\mathcal{T}}(X)$ is a $\mathcal{T}$-strong set. Let $j\in [m]$ be the smallest index such that $X_j$ is not contained in $\fcl_{\mathcal{T}}(X)$. Now, the partition $((X\cup (\bigcup_{i=1}^j X_i))\cap \fcl_{\mathcal{T}}(X), E-((X\cup (\bigcup_{i=1}^j X_i))\cap \fcl_{\mathcal{T}}(X)))$ of $E$ is $\mathcal{T}$-strong because $(X\cup (\bigcup_{i=1}^j X_i))\cap \fcl_{\mathcal{T}}(X)$ contains the $\mathcal{T}$-strong set $X$ while $E-((X\cup (\bigcup_{i=1}^j X_i))\cap \fcl_{\mathcal{T}}(X))$ contains the $\mathcal{T}$-strong set $E-\fcl_{\mathcal{T}}(X)$. Thus $\lambda((X\cup (\bigcup_{i=1}^j X_i))\cap \fcl_{\mathcal{T}}(X))\geq k$ by (T2). Then uncrossing the $k$-separating sets $X\cup (\bigcup_{i=1}^j X_i)$ and $\fcl_{\mathcal{T}}(X)$ we see that $\fcl_{\mathcal{T}}(X)\cup X_j$ is $k$-separating; a contradiction because $\fcl_{\mathcal{T}}(X)$ is fully closed.
\end{proof}

We have the following immediate corollary of Lemma \ref{fclcontain} and Lemma \ref{fcloperator}.

\begin{cor}
 \label{fclcontaincor}
Let $\mathcal{T}$ be a tangle of order $k$ in a connectivity system $(E,\lambda)$, and let $X$ be a $\mathcal{T}$-strong $k$-separating set in $\lambda$. If $(X_i)_{i=1}^m$ is a partial $k$-sequence for $X$, then $\fcl_{\mathcal{T}}(X\cup (\bigcup_{i=1}^m X_i))=\fcl_{\mathcal{T}}(X)$. 
\end{cor}

Let $X$ be a $\mathcal{T}$-strong $k$-separating set in $\lambda$. Let $$P=\textstyle{\{X\cup (\bigcup_{i=1}^m X_i)\ |\ (X_i)_{i=1}^m \text{ is a partial $k$-sequence for $X$ }\}}.$$ Then it is easy to see that $(P,\subseteq)$ is a poset. A partial $k$-sequence $(X_i)_{i=1}^m$ for $X$ is said to be \textit{maximal} if $X\cup (\bigcup_{i=1}^m X_i)$ is maximal in the poset $(P,\subseteq)$. We next characterise the full closure in terms of partial $k$-sequences.

\begin{lemma}
\label{maxsequence}
Let $\mathcal{T}$ be a tangle of order $k$ in a connectivity system $(E,\lambda)$. Let $X$ be a $\mathcal{T}$-strong $k$-separating set, and let $(X_i)_{i=1}^m$ be a partial $k$-sequence for $X$. Then $\fcl_{\mathcal{T}}(X)=X\cup (\bigcup_{i=1}^m X_i)$ if and only if $(X_i)_{i=1}^m$ is maximal.
\end{lemma}

\begin{proof}
Assume that $\fcl_{\mathcal{T}}(X)=X\cup (\bigcup_{i=1}^m X_i)$. Then $(X_i)_{i=1}^m$ is maximal by Lemma \ref{fclcontain}. Conversely, assume that $(X_i)_{i=1}^m$ is maximal. Then it follows from Lemma \ref{fclcontain} that $X\cup (\bigcup_{i=1}^m X_i)\subseteq \fcl_{\mathcal{T}}(X)$. We claim that $X\cup (\bigcup_{i=1}^m X_i)$ is fully closed. Suppose that $X\cup (\bigcup_{i=1}^m X_i)$ is not fully closed. Then there is a non-empty $\mathcal{T}$-weak subset $X_{m+1}$ of $E-(X\cup (\bigcup_{i=1}^m X_i))$ such that $X\cup (\bigcup_{i=1}^{m+1} X_i)$ is $k$-separating. Hence $(X_i)_{i=1}^{m+1}$ is a partial $k$-sequence, and $X\cup (\bigcup_{i=1}^m X_i)\subsetneq X\cup (\bigcup_{i=1}^{m+1} X_i)$; a contradiction because $(X_i)_{i=1}^m$ is maximal. Thus $X\cup (\bigcup_{i=1}^m X_i)$ is a fully-closed $k$-separating set that contains $X$, so $\fcl_{\mathcal{T}}(X)\subseteq X\cup (\bigcup_{i=1}^m X_i)$. Therefore $\fcl_{\mathcal{T}}(X)=X\cup (\bigcup_{i=1}^m X_i)$, as required. 
\end{proof}


We can extend the relation of $\mathcal{T}$-equivalence to the set of 
$\mathcal{T}$-strong $k$-separations of $\lambda$ in the natural way. 
Let $(X,Y)$ and $(X',Y')$ be $\mathcal{T}$-strong $k$-separations of $\lambda$. 
Then $(X,Y)$ is $\mathcal{T}$-\textit{equivalent} to $(X',Y')$ if 
$\{\fcl_{\mathcal{T}}(X),\fcl_{\mathcal{T}}(Y)\}= 
\{\fcl_{\mathcal{T}}(X'),\fcl_{\mathcal{T}}(Y')\}$. 
It is easy to see that $\mathcal{T}$-equivalence is an equivalence relation 
on the set of $\mathcal{T}$-strong $k$-separations of $\lambda$.  When the tangle 
$\mathcal{T}$ is clear from the context, we shall abbreviate 
``$\mathcal{T}$-equivalent'' to ``equivalent''.

Let $X$ be a $k$-separating set in $\lambda$. We say that $X$ is $\mathcal{T}$-\textit{sequential} if $E-X$ is $\mathcal{T}$-strong and $\fcl_{\mathcal{T}}(E-X)=E$. A $k$-separation $(X,Y)$ is $\mathcal{T}$-\textit{sequential} if $X$ or $Y$ is a $\mathcal{T}$-sequential $k$-separating set. When the tangle $\mathcal{T}$ is clear from the context, we shall use ``sequential'' and ``non-sequential'' instead of ``$\mathcal{T}$-sequential'' and ``not $\mathcal{T}$-sequential'' respectively. It is clear that every non-sequential $k$-separation must be a $\mathcal{T}$-strong $k$-separation. 

The remainder of this section is devoted to developing some useful lemmas about 
$k$-separations of a connectivity function $\lambda$ that are 
$\mathcal{T}$-equivalent with respect to a tangle $\mathcal{T}$ of order $k$.

The following lemma provides an economical test of equivalence for 
non-sequential $k$-separations.

\begin{lemma}
 \label{3.3}
Let $\mathcal{T}$ be a tangle of order $k$ in a connectivity system $(E,\lambda)$, and let $(A, B)$ and $(C, D)$ be two non-sequential $k$-separations of $\lambda$. Then $(A, B)$ is $\mathcal{T}$-equivalent to $(C, D)$ if and only if either $\fcl_{\mathcal{T}}(A)= \fcl_{\mathcal{T}}(C)$ or $\fcl_{\mathcal{T}}(A)= \fcl_{\mathcal{T}}(D)$.
\end{lemma}

\begin{proof}
In one direction the lemma is trivial. For the other direction, assume that $\fcl_{\mathcal{T}}(A)= \fcl_{\mathcal{T}}(C)=Y$. Set $X=E-Y$. Then $X$ is a $\mathcal{T}$-strong $k$-separating set because $(A,B)$ is a non-sequential $k$-separation. Let $(A_i)_{i=1}^m$ be a maximal partial $k$-sequence for $A$ and $(C_i)_{i=1}^n$ be a maximal partial $k$-sequence for $C$. Then it follows from Lemma \ref{maxsequence} that $A\cup (\bigcup_{i=1}^m A_i)=\fcl_{\mathcal{T}}(A)$ and $C\cup (\bigcup_{i=1}^n C_i) =\fcl_{\mathcal{T}}(C)$. Then $(A_{m-i+1})_{i=1}^m$ and $(C_{n-i+1})_{i=1}^n$ are partial $k$-sequences for $X$, so $B=X\cup (\bigcup_{i=1}^m A_{m-i+1})\subseteq \fcl_{\mathcal{T}}(X)$ and $D=X\cup (\bigcup_{i=1}^n C_{n-i+1})\subseteq \fcl_{\mathcal{T}}(X)$ by Lemma \ref{fclcontain}. By Lemma \ref{fcloperator}, and since $X$ is a subset of both $B$ and $D$, we have $\fcl_{\mathcal{T}}(B)= \fcl_{\mathcal{T}}(X)= \fcl_{\mathcal{T}}(D)$, so $(A,B)$ is indeed equivalent to $(C,D)$. 
\end{proof}

The following lemma contains some elementary results about $\mathcal{T}$-equivalence 
of $\mathcal{T}$-strong $k$-separations. It is used frequently.

\begin{lemma}
\label{bigeqlemma}
Let $\mathcal{T}$ be a tangle of order $k$ in a connectivity system $(E,\lambda)$, and let $(R,G)$ be a $\mathcal{T}$-strong $k$-separation of $\lambda$.
\begin{enumerate}
\item[(i)] If $A\subseteq G$ is a non-empty $\mathcal{T}$-weak set such that $R\cup A$ is $k$-separating, and $G-A$ is $\mathcal{T}$-strong, then $(R,G)$ is $\mathcal{T}$-equivalent to $(R\cup A, G-A)$.
\item[(ii)]If $(R,G)$ is a non-sequential $k$-separation, and $A\subseteq G$ is a non-empty $\mathcal{T}$-weak set such that $R\cup A$ is $k$-separating, then $(R\cup A, G-A)$ is $\mathcal{T}$-equivalent to $(R,G)$.
\item[(iii)] If $(R,G)$ is a non-sequential $k$-separation, then $(\fcl_{\mathcal{T}}(R), E-\fcl_{\mathcal{T}}(R))$ is $\mathcal{T}$-equivalent to $(R,G)$.  
\item[(iv)] If $(R,G)$ is a non-sequential $k$-separation, and $X$ is a $k$-separating set such that $E-\fcl_{\mathcal{T}}(G)\subseteq X\subseteq R$, then $(X,E-X)$ is $\mathcal{T}$-equivalent to $(R,G)$.
\end{enumerate}
\end{lemma}

\begin{proof}
For (i), suppose that $A\subseteq G$ is a non-empty $\mathcal{T}$-weak set such that $R\cup A$ is $k$-separating, and that $G-A$ is $\mathcal{T}$-strong. Then $(A)$ is a partial $k$-sequence for both $R$ and $G-A$, so  $\fcl_{\mathcal{T}}(R)=\fcl_{\mathcal{T}}(R\cup A)$ and $\fcl_{\mathcal{T}}(G-A)=\fcl_{\mathcal{T}}(G)$ by Corollary \ref{fclcontaincor}.

To prove (ii), we show that $G-A$ is $\mathcal{T}$-strong and then apply (i). Assume towards a contradiction that $G-A$ is $\mathcal{T}$-weak. Then $(A,G-A)$ is a partial $k$-sequence for $R$, so it follows from Lemma \ref{fclcontain} that $\fcl_{\mathcal{T}}(R)=E$; a contradiction because $(R,G)$ is non-sequential.

For (iii), suppose that $(R_i)_{i=1}^m$ is a partial $k$-sequence for $R$. Then it follows from (ii) and induction that $(R\cup (\bigcup_{i=1}^m R_i), G-(\bigcup_{i=1}^m R_i))$ is $\mathcal{T}$-equivalent to $(R,G)$. In particular, if $(R_i)_{i=1}^m$ is a maximal partial $k$-sequence, then $\fcl_{\mathcal{T}}(R)= R\cup (\bigcup_{i=1}^m R_i)$ by Lemma \ref{maxsequence}, so $(\fcl_{\mathcal{T}}(R), E-\fcl_{\mathcal{T}}(R))$ is $\mathcal{T}$-equivalent to $(R,G)$. 

To prove (iv), we first observe that $\fcl_{\mathcal{T}}(E-\fcl_{\mathcal{T}}(G))=\fcl_{\mathcal{T}}(R)$ by (iii). Thus from $E-\fcl_{\mathcal{T}}(G)\subseteq X\subseteq R$ and Lemma \ref{fcloperator}, it follows that $\fcl_{\mathcal{T}}(X)=\fcl_{\mathcal{T}}(R)$. Furthermore, from the fact that $E-\fcl_{\mathcal{T}}(G)\subseteq X\subseteq R$ it follows that $G\subseteq E-X\subseteq \fcl_{\mathcal{T}}(G)$, so $\fcl_{\mathcal{T}}(E-X)=\fcl_{\mathcal{T}}(G)$ by Lemma \ref{fcloperator}.
\end{proof}

Our primary aim is to display all non-sequential $k$-separations in a tree structure, but we can work with a more specific collection of non-sequential separations with little additional cost.

Let $\mathcal{S}$ be a set of non-sequential $k$-separating sets in $\lambda$ with $\mathcal{T}$-strong complements, and let $(X,E-X)$ be a $k$-separation of $\lambda$. We say that $(X,E-X)$ is a $(k, \mathcal{S})$-\textit{separation} if $X, E-X\in \mathcal{S}$. The set $\mathcal{S}$ is said to be \textit{tree compatible} if the following hold:

\begin{enumerate}
 \item[(S1)] If $(X,E-X)$ is a $(k, \mathcal{S})$-separation and $(Y,E-Y)$ is a $\mathcal{T}$-strong $k$-separation that is $\mathcal{T}$-equivalent to $(X,E-X)$, then $(Y,E-Y)$ is a $(k, \mathcal{S})$-separation.
 \item[(S2)] If $X\in \mathcal{S}$, and $(Y,E-Y)$ is a $\mathcal{T}$-strong $k$-separation such that $X\subseteq Y$, then $Y\in \mathcal{S}$.
\end{enumerate}

For example, it is not difficult to see that the set of all non-sequential $k$-separating sets in $\lambda$ with $\mathcal{T}$-strong complements is tree compatible. In this case, a $(k, \mathcal{S})$-separation is simply a non-sequential $k$-separation. 

We frequently make use of the fact that a $(k, \mathcal{S})$-separation is, in particular, a non-sequential $k$-separation.

\subsection*{Vertically $k$-connected matroids} Let $M$ be a vertically $k$-connected
matroid whose rank is at least $\max\{3k-5,2\}$. Recall that $\tangle_k$ denotes the
unique tangle of order $k$ in $M$. Recall that a set $X$ is $\tangle_k$ strong if 
$r(X)\geq k-1$. Thus any exactly $k$-separating set is $\tangle_k$-strong. If 
$X$ is exactly $k$-separating, then $X$ is fully closed relative to $\tangle_k$
if and only if $\lambda(X\cup Y)>k$ for every subset $Y$ of $E(M)-X$ with $r(Y)\leq k-2$. 
A partial $k$-sequence for $X$ is a sequence $(X_i)_{i=1}^m$ of pairwise disjoint,
non-empty subsets such that $r(X_j)\leq k-2$ and $X\cup(\bigcup_{i=1}^j X_i)$ is 
$k$-separating for all $j\in[m]$. The same interpretations apply when $M$ is strictly
$k$-connected except that we replace rank by cardinality in the statements. This leads to 
a notion of equivalence for $k$-separating sets and $k$-separations
in vertically $k$-connected matroids
or strictly $k$-connected matroids.  When $k=3$ and $M$ is $3$-connected, 
this notions of equivalence is precisely
the equivalence of $3$-separations 
defined by Oxley, Semple, and Whittle \cite{oxley2004structure}. When $k=4$ and $M$
is 4-connected it is precisely the equivalence 
of $4$-separations defined by Aikin and Oxley \cite{aikin2011structure}.
All of the lemmas of this section now have obvious 
specialisation when interpreted for $\tangle_k$.

\section{Flowers in a tangle}
\label{flowtang}

Let $\mathcal{T}$ be a tangle of order $k$ in a connectivity system $(E,\lambda)$, and let $n$ be a positive integer. A $\mathcal{T}$-strong partition $(P_1,\ldots, P_n)$ of $E$ is a \textit{$k$-flower} in $\mathcal{T}$ with \textit{petals} $P_1,\ldots, P_n$ if, for all $i$, both $P_i$ and $P_i\cup P_{i+1}$ are $k$-separating sets, where all subscripts are interpreted modulo $n$. 

We next define some of the fundamental notions for $k$-flowers in $\mathcal{T}$. Most of these are natural extensions of the analogous notions for flowers in $3$-connected matroids given in \cite{oxley2004structure}.

Let $\mathcal{T}$ be a tangle of order $k$ in a connectivity system $(E,\lambda)$, and let $\Phi=(P_1,\ldots, P_n)$ be a $k$-flower in $\mathcal{T}$. A $k$-separating set $X$ or $k$-separation $(X,E-X)$ is said to be \textit{displayed} by $\Phi$ if $X$ is a union of petals of $\Phi$. For a  non-empty subset $I$ of $[n]$, we write $P_I$ for $\bigcup_{i\in I} P_i$. A $k$-flower $\Phi=(P_1,\ldots, P_n)$ in $\mathcal{T}$ is called a \textit{$k$-anemone} if $P_I$ is $k$-separating for any non-empty subset $I$ of $[n]$, and a \textit{$k$-daisy} if $P_I$ is $k$-separating for precisely those non-empty subsets $I$ of $[n]$ whose members form a consecutive set in the cyclic order $(1,\ldots, n)$. As every $k$-flower in $\mathcal{T}$ is a $k$-flower in the connectivity function $\lambda$ (see \cite{aikin2008structure}), we have the following immediate consequence of \cite[Theorem 1.1]{aikin2008structure}.

\begin{cor}
\label{flowerclass}
Every $k$-flower in $\mathcal{T}$ is either a $k$-anemone or a $k$-daisy.
\end{cor}

Let $\mathcal{T}$ be a tangle of order $k$ in a connectivity system $(E,\lambda)$, and let $\mathcal{S}$ be a tree compatible set. We define a relation $\preccurlyeq_{\mathcal{S}}$ on the set of $k$-flowers in $\mathcal{T}$ as follows. Let $\Phi_1$ and $\Phi_2$ be $k$-flowers in $\mathcal{T}$. We say that $\Phi_1 \preccurlyeq_{\mathcal{S}} \Phi_2$ if, for every $(k, \mathcal{S})$-separation displayed by $\Phi_1$, there is some $\mathcal{T}$-equivalent $(k, \mathcal{S})$-separation displayed by $\Phi_2$. It is straightforward to verify that the relation $\preccurlyeq_{\mathcal{S}}$ is a quasi-order on the set of $k$-flowers in $\mathcal{T}$. If $\Phi_1 \preccurlyeq_{\mathcal{S}} \Phi_2$ and $\Phi_2 \preccurlyeq_{\mathcal{S}} \Phi_1$, we say that $\Phi_1$ and $\Phi_2$ are \textit{$\mathcal{T}$-equivalent} $k$-flowers with respect to $\mathcal{S}$. Thus, $k$-flowers that are $\mathcal{T}$-equivalent with respect to $\mathcal{S}$ display, up to $\mathcal{T}$-equivalence of $k$-separations, exactly the same $(k, \mathcal{S})$-separations of $\lambda$. Note that when the tangle $\mathcal{T}$ and the set $\mathcal{S}$ are clear from the context, we shall abbreviate ``$\mathcal{T}$-equivalent with respect to $\mathcal{S}$'' to ``equivalent''.

Let $\Phi=(P_1,\ldots, P_n)$ be a $k$-flower in $\mathcal{T}$, and let $\mathcal{S}$ be a tree compatible set. If $\Phi$ is a $k$-anemone and $\sigma$ is an arbitrary permutation of the set $[n]$, then it is easy to see that $\Phi '=(P_{\sigma(1)},\ldots, P_{\sigma(n)})$ is a $k$-flower in $\mathcal{T}$ that is equivalent to $\Phi$. Similarly, if $\Phi=(P_1,\ldots, P_n)$ is a $k$-daisy and $\sigma$ is a permutation of the set $[n]$ that corresponds to a symmetry of a regular $n$-gon, then it is easy to see that $\Phi '=(P_{\sigma(1)},\ldots, P_{\sigma(n)})$ is a $k$-flower in $\mathcal{T}$ that is equivalent to $\Phi$. We say that $\Phi$ and $\Phi '$ are \textit{equal up to labels}\index{up to labels}. We will often use the phrase ``up to labels'' to mean ``by an appropriate permutation of the petals''.

We now describe a fundamental method of obtaining new $k$-flowers in $\mathcal{T}$ from old. Let $\Phi=(P_1,P_2,\ldots, P_n)$ be a $k$-flower in $\mathcal{T}$. Then the ordered partition $\Phi '=(P_1',P_2',\ldots,P_m')$ is a \textit{concatenation} of $\Phi$ if there are indices $0=j_0<j_1<j_2<\cdots <j_m=m$ such that $P_i'=P_{j_{i-1}+1}\cup \dots \cup P_{j_i}$ for all $i\in [m]$. It is not difficult to prove that, if $\Phi'$ is a concatenation of a $k$-flower $\Phi$ in $\mathcal{T}$, then $\Phi'$ is also a $k$-flower in $\mathcal{T}$. If $\Phi '$ is a concatenation of $\Phi$, then we say that $\Phi$ \textit{refines} $\Phi '$.

The following is an economical way to show that a $\mathcal{T}$-strong partition of $E$ is a $k$-flower in $\mathcal{T}$.

\begin{lemma}
 \label{flowerunion}
Let $\mathcal{T}$ be a tangle of order $k$ in a connectivity system $(E,\lambda)$. Let $n\geq 4$, and let $\Phi=(P_1,\ldots,P_n)$ be a $\mathcal{T}$-strong partition of $E$. If $P_i\cup P_{i+1}$ is $k$-separating for each $i\in [n-1]$, then $\Phi$ is a $k$-flower in $\mathcal{T}$.
\end{lemma}

\begin{proof}
Suppose that $P_i\cup P_{i+1}$ is $k$-separating for each $i\in [n-1]$. We first show that $P_n\cup P_1$ is $k$-separating. If $n=4$, then $P_n\cup P_1$ is $k$-separating because $P_2\cup P_3$ is $k$-separating and $\lambda$ is symmetric. For $n>4$, the sets $P_2\cup P_3$ and $P_3\cup P_4$ are $k$-separating, and their intersection is the set $P_3$. Since the partition $(P_3, E-P_3)$ is $\mathcal{T}$-strong, we deduce that $\lambda(P_3)\geq k$ by (T2). Then $P_2\cup P_3\cup P_4$ is $k$-separating by uncrossing $P_2\cup P_3$ and $P_3\cup P_4$. By repeated uncrossings, we deduce that $P_2\cup \cdots \cup P_{n-1}$ is $k$-separating, so the complement $E-(P_2\cup \cdots \cup P_{n-1})=P_n\cup P_1$ is $k$-separating. Thus the union of any two members of $\Phi$ that are consecutive in the cyclic order is $k$-separating. Furthermore, for each $i\in [n]$, the set $P_i$ is the intersection of the $k$-separating sets $P_{i-1}\cup P_i$ and $P_i\cup P_{i+1}$, where all subscripts are interpreted modulo $n$. Since $n\geq 4$, the partition $(P_{i-1}\cup P_i\cup P_{i+1}, E-(P_{i-1}\cup P_i\cup P_{i+1}))$ is $\mathcal{T}$-strong, so $\lambda(E-(P_{i-1}\cup P_i\cup P_{i+1}))\geq k$ by (T2). Thus $P_i$ is $k$-separating by uncrossing $P_{i-1}\cup P_i$ and $P_i\cup P_{i+1}$. Therefore $\Phi$ is a $k$-flower in $\mathcal{T}$.
\end{proof}

Let $\mathcal{T}$ be a tangle of order $k$ in a connectivity system $(E,\lambda)$, and let $\mathcal{S}$ be a tree compatible set. The $\mathcal{S}$-\textit{order} of a $k$-flower $\Phi$ in $\mathcal{T}$ is the minimum number of petals in a $k$-flower that is $\mathcal{T}$-equivalent to $\Phi$ with respect to $\mathcal{S}$. Up to equivalence of $k$-separations, a $k$-flower of $\mathcal{S}$-order one displays no $(k, \mathcal{S})$-separations, a $k$-flower of $\mathcal{S}$-order two displays exactly one $(k, \mathcal{S})$-separation, and a $k$-flower of $\mathcal{S}$-order at least three displays at least two $(k, \mathcal{S})$-separations. A $k$-flower is $\mathcal{S}$-\textit{tight} if it is not equivalent to a $k$-flower with fewer petals.

Let $n\geq 2$, and let $\Phi=(P_1,\ldots,P_n)$ be a $k$-flower in $\mathcal{T}$. For $i\in [n]$, the petal $P_i$ of $\Phi$ is \textit{$\mathcal{T}$-loose} if $P_i\subseteq \fcl_{\mathcal{T}}(P_j)$ for some petal $P_j\neq P_i$ such that, up to labels, $P_i$ and $P_j$ are consecutive in the cyclic order on $\Phi$. The next result shows that $k$-flowers in $\mathcal{T}$ that have $\mathcal{T}$-loose petals are not $\mathcal{S}$-tight. 

\begin{lemma}
\label{flowereq}
Let $\mathcal{T}$ be a tangle of order $k$ in a connectivity system $(E,\lambda)$, and let $\mathcal{S}$ be a tree compatible set. Let $n\geq 2$, and let $\Phi=(P_1,\ldots,P_n)$ be a $k$-flower in $\mathcal{T}$. If $P_1\subseteq \fcl_{\mathcal{T}}(P_2)$, then the concatenation $\Phi '=(P_1\cup P_2,P_3,\ldots, P_n)$ of $\Phi$ is $\mathcal{T}$-equivalent to $\Phi$ with respect to $\mathcal{S}$. 
\end{lemma}

\begin{proof}
Assume that $P_1\subseteq \fcl_{\mathcal{T}}(P_2)$, and let $\Phi '=(P_1\cup P_2,P_3,\ldots, P_n)$. Since $\Phi$ refines $\Phi '$, it follows immediately that $\Phi '\preccurlyeq_{\mathcal{S}} \Phi$. Assume that $(R,G)$ is a $(k, \mathcal{S})$-separation displayed by $\Phi$. If $P_1$ and $P_2$ are both contained in either $R$ or $G$, then $(R,G)$ is displayed by $\Phi '$. Thus we may assume, up to switching $R$ and $G$, that $P_1\subseteq G$ and $P_2\subseteq R$. We claim that the partition $(R\cup P_1,G-P_1)$ is a $(k, \mathcal{S})$-separation that is $\mathcal{T}$-equivalent to $(R,G)$. The set $R\cup P_1$ is the union of the $k$-separating sets $P_1\cup P_2$ and $R$, and their intersection is $P_2$, so $R\cup P_1$ is $k$-separating by uncrossing $P_1\cup P_2$ and $R$. Moreover, $G-P_1$ contains some petal of $\Phi$ because $(R,G)$ is a non-sequential $k$-separation and $P_1\subseteq \fcl_{\mathcal{T}}(P_2)\subseteq \fcl_{\mathcal{T}}(R)$ by Lemma \ref{fcloperator}. Thus $(R\cup P_1,G-P_1)$ is a $\mathcal{T}$-strong $k$-separation. Now from Lemma \ref{fcloperator} it follows that $\fcl_{\mathcal{T}}(R\cup P_1)=\fcl_{\mathcal{T}}(R)$ and $\fcl_{\mathcal{T}}(G-P_1)\subseteq \fcl_{\mathcal{T}}(G)$, so $(R\cup P_1,G-P_1)$ is a non-sequential $k$-separation because $(R,G)$ is a non-sequential $k$-separation. Hence $(R\cup P_1,G-P_1)$ $\mathcal{T}$-equivalent to $(R,G)$ by the fact that $\fcl_{\mathcal{T}}(R\cup P_1)=\fcl_{\mathcal{T}}(R)$ and Lemma \ref{3.3}, so $(R\cup P_1,G-P_1)$ is a $(k,\mathcal{S})$-separation by (S1). Then $\Phi\preccurlyeq_{\mathcal{S}}\Phi '$ because $(R\cup P_1,G-P_1)$ is displayed by $\Phi '$. Thus $\Phi$ and $\Phi '$ are indeed $\mathcal{T}$-equivalent $k$-flowers with respect to $\mathcal{S}$. 
\end{proof}

\begin{lemma}
 \label{nompetalt}
Let $\mathcal{T}$ be a tangle of order $k$ in a connectivity system $(E,\lambda)$, and let $\mathcal{S}$ be a tree compatible set. If $\Phi=(P_1,P_2,\ldots, P_n)$ is a $k$-flower in $\mathcal{T}$ of $\mathcal{S}$-order at least two, then $E-P_i\in \mathcal{S}$ for all $i\in [n]$.
\end{lemma}

\begin{proof}
Assume that $\Phi=(P_1,P_2,\ldots, P_n)$ is a $k$-flower in $\mathcal{T}$ of $\mathcal{S}$-order at least two. Then $\Phi$ displays at least one $(k, \mathcal{S})$-separation $(R,G)$, so either $P_i\in R$ or $P_i\in G$ for each $i\in [n]$. Then $E-P_i$ must contain $R$ or $G$, so $E-P_i\in \mathcal{S}$ by (S2). 
\end{proof}

\begin{lemma}
\label{tightflowerwin}
Let $\mathcal{T}$ be a tangle of order $k$ in a connectivity system $(E,\lambda)$, and let $\mathcal{S}$ be a tree compatible set. Let $\Phi=(P_1,P_2,\ldots, P_n)$ be a $k$-flower in $\mathcal{T}$ of $\mathcal{S}$-order at least two with no $\mathcal{T}$-loose petals. If $X\subseteq E-P_1$ is a non-empty $\mathcal{T}$-weak set such that $P_1\cup X$ is $k$-separating, then $P_i-X$ is $\mathcal{T}$-strong for all $i\in [2,n]$.
\end{lemma}

\begin{proof}
Let $X\subseteq E-P_1$ be a non-empty $\mathcal{T}$-weak set such that $P_1\cup X$ is $k$-separating. 

\begin{claim}
\label{tifw1}
$P_1\cup P_2\cup \cdots \cup P_i\cup X$ is $k$-separating for all $i\in [2,n-1]$. 
\end{claim}

\begin{subproof}
Let $i\in [2,n-1]$. The set $P_1\cup P_2\cup \cdots \cup P_i\cup X$ is the union of the $k$-separating sets $P_1\cup P_2\cup \cdots \cup P_i$ and $P_1\cup X$. The set $P_1\cup (X\cap (P_1\cup P_2\cup \cdots \cup P_i))$ contains the $\mathcal{T}$-strong set $P_1$ and avoids the $\mathcal{T}$-strong set $P_n$, so the partition $$(P_1\cup (X\cap (P_1\cup P_2\cup \cdots \cup P_i)), ((P_2\cup \cdots \cup P_i)-X)\cup P_{i+1}\cup \cdots \cup P_n)$$ is $\mathcal{T}$-strong, and so $\lambda(P_1\cup (X\cap (P_1\cup P_2\cup \cdots \cup P_i)))\geq k$ by (T2). Since the set $P_1\cup (X\cap (P_1\cup P_2\cup \cdots \cup P_i))$ is the intersection of $P_1\cup P_2\cup \cdots \cup P_i$ and $P_1\cup X$, the set $P_1\cup P_2\cup \cdots \cup P_i\cup X$ is $k$-separating by uncrossing $P_1\cup P_2\cup \cdots \cup P_i$ and $P_1\cup X$.  
\end{subproof}

Assume that $P_i-X$ is $\mathcal{T}$-weak for all $i\in [2,n]$. Then, by \ref{tifw1}, the sequence $(X, P_2-X, \ldots, P_n-X)$ is a partial $k$-sequence for $P_1$. Thus $X\cup (P_2-X)\cup \cdots \cup (P_n-X)\subseteq \fcl_{\mathcal{T}}(P_1)$ by Lemma \ref{fclcontain}, and so $\fcl_{\mathcal{T}}(P_1)=E$; a contradiction because $E-P_1$ is a non-sequential $k$-separating set by Lemma \ref{nompetalt}. We may therefore assume that there is some $j\in [2,n]$ such that $P_j-X$ is $\mathcal{T}$-strong. Seeking a contradiction, suppose that $P_i-X$ is $\mathcal{T}$-weak for some $i\in [2,n]$. We may assume, by reversing the order of the petals $P_2,\ldots, P_n$ of $\Phi$ if necessary, that $i\in [2,j-1]$. Now, the set $P_{i-1}\cup (P_i\cap X)$ is the intersection of the $k$-separating sets $P_1\cup \cdots \cup P_{i-1}\cup X$ and $P_{i-1}\cup P_i$. Since $P_1\cup P_2\cup \cdots \cup P_i\cup X$, the union of $P_1\cup \cdots \cup P_{i-1}\cup X$ and $P_{i-1}\cup P_i$, is $\mathcal{T}$-strong and avoids the $\mathcal{T}$-strong set $P_j-X$, the partition $$(P_1\cup P_2\cup \cdots \cup P_i\cup X, E-(P_1\cup P_2\cup \cdots \cup P_i\cup X))$$ is $\mathcal{T}$-strong, and so $\lambda(P_1\cup P_2\cup \cdots \cup P_i\cup X)\geq k$ by (T2). Then $P_{i-1}\cup (P_i\cap X)$ is $k$-separating by uncrossing $P_1\cup \cdots \cup P_{i-1}\cup X$ and $P_{i-1}\cup P_i$. But then $(P_i\cap X, P_i-X)$ is a partial $k$-sequence for $P_{i-1}$, so $P_i\subseteq \fcl_{\mathcal{T}}(P_{i-1})$ by Lemma \ref{fclcontain}; a contradiction because $\Phi$ has no $\mathcal{T}$-loose petals.
\end{proof} 

The next lemma relates equivalence of $k$-separations and equivalence of $k$-flowers in $\mathcal{T}$.

\begin{lemma}
\label{flowereq3}
Let $\mathcal{T}$ be a tangle of order $k$ in a connectivity system $(E,\lambda)$, and let $\mathcal{S}$ be a tree compatible set. Let $\Phi=(P_1,\ldots, P_n)$ be a $k$-flower in $\mathcal{T}$ of $\mathcal{S}$-order at least two with no $\mathcal{T}$-loose petals. If $X\subseteq E-P_1$ is a non-empty $\mathcal{T}$-weak set such that $P_1\cup X$ is $k$-separating, then: 
\begin{enumerate}
 \item[(i)] $\Phi '=(P_1\cup X, P_2-X,\ldots, P_n-X)$ is a $k$-flower in $\mathcal{T}$ that is $\mathcal{T}$-equivalent to $\Phi$ with respect to $\mathcal{S}$.
 \item[(ii)] $\fcl_{\mathcal{T}}(P_1)=\fcl_{\mathcal{T}}(P_1\cup X)$ and $\fcl_{\mathcal{T}}(P_i-X)=\fcl_{\mathcal{T}}(P_i)$ for all $i\in [2,n]$.
\end{enumerate}
\end{lemma}

\begin{proof}
Suppose that $X\subseteq E-P_1$ is a non-empty $\mathcal{T}$-weak set such that $P_1\cup X$ is $k$-separating. We first show that $\Phi '$ is a $k$-flower in $\mathcal{T}$. It follows from Lemma \ref{tightflowerwin} that the partition $\Phi '$ is $\mathcal{T}$-strong. It remains to show that each member of $\Phi '$ is $k$-separating and that the union of any two consecutive members of $\Phi '$ is $k$-separating. If $n=2$, then this follows from Lemma \ref{bigeqlemma} (ii). Assume that $n=3$. Then, for $i\in \{2,3\}$, it follows from uncrossing $E-P_i$ and $P_1\cup X$ that $(P_1\cup X)\cup (P_j-X)$ is $k$-separating for $j\in \{2,3\}-\{i\}$. Thus $\Phi '$ is a $k$-flower. We may therefore assume that $n\geq 4$. Now, the set $P_1\cup X \cup (P_2-X)$ is the union of the $k$-separating sets $P_1\cup P_2$ and $P_1\cup X$ whose intersection is $P_1\cup (P_2\cap X)$. The partition $(P_1\cup (P_2\cap X), E-(P_1\cup (P_2\cap X)))$ is clearly $\mathcal{T}$-strong, so $\lambda(P_1\cup (P_2\cap X))\geq k$ by (T2). Thus $P_1\cup X \cup (P_2-X)$ is $k$-separating by uncrossing $P_1\cup P_2$ and $P_1\cup X$. Moreover, for each $i\in [2,n-1]$, the set $(P_i-X)\cup (P_{i+1}-X)$ is the intersection of the $k$-separating sets $E-(P_1\cup X)$ and $P_i\cup P_{i+1}$, whose union is a $\mathcal{T}$-strong set that avoids the $\mathcal{T}$-strong set $P_1$, so $\lambda((E-(P_1\cup X))\cup P_i\cup P_{i+1})\geq k$ by (T2). Then $(P_i-X)\cup (P_{i+1}-X)$ is $k$-separating by uncrossing $E-(P_1\cup X)$ and $P_i\cup P_{i+1}$. Thus $\Phi '$ satisfies the hypotheses of Lemma \ref{flowerunion}, so $\Phi '$ is a $k$-flower in $\mathcal{T}$.
 
We now show that $\Phi '$ is $\mathcal{T}$-equivalent to $\Phi$ with respect to $\mathcal{S}$. Suppose that $(R,G)$ is a $(k,\mathcal{S})$-separation displayed by $\Phi$. Then, up to switching $R$ and $G$, we may assume that $P_1\subseteq R$. Then $R\cup X$ is the union of the $k$-separating sets $P_1\cup X$ and $R$, whose intersection is $P_1\cup (R\cap X)$. Since the partition $(P_1\cup (R\cap X), G\cup (R-(P_1\cup X)))$ is $\mathcal{T}$-strong, it follows from (T2) that $\lambda(P_1\cup (R\cap X))\geq k$. Thus $R\cup X$ is $k$-separating by uncrossing $P_1\cup X$ and $R$. Thus $(R\cup X,G-X)$ is a $\mathcal{T}$-equivalent $k$-separation to $(R,G)$ by Lemma \ref{bigeqlemma}(ii), so $(R\cup X,G-X)$ is a $(k,\mathcal{S})$-separation. Moreover, $(R\cup X,G-X)$ is displayed by $\Phi '$. Thus $\Phi \preccurlyeq_{\mathcal{S}} \Phi '$. Now suppose that $(R,G)$ is a $(k,\mathcal{S})$-separation that is displayed by $\Phi '$. Then, up to switching $R$ and $G$, we may assume that $P_1\cup X\subseteq R$. Let $I=\{i\in [2,n]\ |\ P_i-X\subseteq G \}$. Then it is easy to check that $P_I$ is a $k$-separating set. It now follows that $G\cup (X\cap P_I)$ is a $k$-separating set by uncrossing $G$ and $P_I$, so $(G\cup (X\cap P_I), R-(X\cap P_I))$ is a $k$-separation that is $\mathcal{T}$-equivalent to $(R,G)$ by Lemma \ref{bigeqlemma}(ii). Thus $(G\cup (X\cap P_I), R-(X\cap P_I))$ is a $(k,\mathcal{S})$-separation by (S2), and $(G\cup (X\cap P_I), R-(X\cap P_I))$ is displayed by $\Phi$, so $\Phi '\preccurlyeq_{\mathcal{S}}\Phi$. This establishes (i).

For (ii), let $(P_1',\ldots,P_n')=(P_1\cup X, P_2-X,\ldots, P_n-X)$. Then $(X)$ is a partial $k$-sequence for $P_1$, and for all $i\in [2,n]$, we observe that $(P_i\cap X)$ is a partial $k$-sequence for $P_i-X$. Then it follows immediately from Corollary \ref{fclcontaincor} that $\fcl_{\mathcal{T}}(P_i')=\fcl_{\mathcal{T}}(P_i)$ for all $i\in [n]$.
\end{proof}

The next result shows that certain concatenations of tight $k$-flowers in $\mathcal{T}$ have no $\mathcal{T}$-loose petals.

\begin{lemma}
\label{tightconcatcond}
Let $\mathcal{T}$ be a tangle of order $k$ in a connectivity system $(E,\lambda)$, and let $\mathcal{S}$ be a tree compatible set. Let $\Phi=(P_1,\ldots,P_n)$ be an $\mathcal{S}$-tight $k$-flower in $\mathcal{T}$ of $\mathcal{S}$-order at least three, and let $j\in [2,n-1]$. If $(P_1\cup \cdots \cup P_j, P_{j+1}\cup \cdots \cup P_n)$ is a $(k, \mathcal{S})$-separation, then the concatenation $\Phi '=(P_1\cup \cdots \cup P_j,  P_{j+1},\ldots, P_n)$ of $\Phi$ has no $\mathcal{T}$-loose petals.
\end{lemma}

\begin{proof}
Suppose that $(P_1\cup \cdots \cup P_j, P_{j+1}\cup \cdots \cup P_n)$ is a $(k,\mathcal{S})$-separation. Let $J=[j]$, and let $\Phi '=(P_{J},  P_{j+1},\ldots, P_n)$. If $j=n-1$, then the lemma immediately holds, so we may assume that $j<n-1$. Seeking a contradiction, suppose that $\Phi '$ has a $\mathcal{T}$-loose petal. Then $\Phi$ has no $\mathcal{T}$-loose petals by Lemma \ref{flowereq}, so we may assume that $P_{j+1}\subseteq \fcl_{\mathcal{T}}(P_{J})$. Let $(X_i)_{i=1}^m$ be a maximal partial $k$-sequence for $P_{J}$. Then $P_{j+1}\subseteq \bigcup_{i=1}^m X_i$. The partition $(P_{J}\cup P_{j+1}\cup X_1, E-(P_{J}\cup P_{j+1}\cup X_1))$ is $\mathcal{T}$-strong because $P_{J}\cup P_{j+1}\cup X_1$ contains the $\mathcal{T}$-strong set $P_J$ and $E-(P_{J}\cup P_{j+1}\cup X_1)$ contains the $\mathcal{T}$-strong set $E-\fcl_{\mathcal{T}}(P_{J})$. Thus $\lambda(P_{J}\cup P_{j+1}\cup X_1)\geq k$ by (T2). It now follows from uncrossing the sets $P_{J}\cup X_1$ and $P_j\cup P_{j+1}$, whose union is $P_{J}\cup P_{j+1}\cup X_1$, that $P_j\cup (P_{j+1}\cap X_1)$ is $k$-separating. This process can clearly be repeated, so that $P_j\cup (\bigcup_{i=1}^{\ell} (P_{j+1}\cap X_i))$ is $k$-separating for all $\ell\in [m]$. Then, up to removing any empty terms, the sequence $(P_{j+1}\cap X_i)_{i=1}^m$ is a partial $k$-sequence for $P_j$, so $P_{j+1}\subseteq \fcl_{\mathcal{T}}(P_j)$ by Lemma \ref{fclcontain}; a contradiction because $\Phi$ has no $\mathcal{T}$-loose petals. 
\end{proof}

The following lemma is useful for locating $(k,\mathcal{S})$-separations displayed by a $k$-flower in $\mathcal{T}$.

\begin{lemma}
\label{confconcatcond}
Let $\mathcal{T}$ be a tangle of order $k$ in a connectivity system $(E,\lambda)$, and let $\mathcal{S}$ be a tree compatible set. Let $\Phi=(P_1,\ldots,P_n)$ be a $k$-flower in $\mathcal{T}$ of $\mathcal{S}$-order at least three, and let $(A,B)$ and $(C,D)$ be inequivalent $(k,\mathcal{S})$-separations of $\lambda$ that are displayed by $\Phi$. If $(A',B')$ is a $(k,\mathcal{S})$-separation that is $\mathcal{T}$-equivalent to $(A,B)$ with $\fcl_{\mathcal{T}}(A)=\fcl_{\mathcal{T}}(A')$, and $A'\subseteq C$, then there is a $(k,\mathcal{S})$-separation $(A'',B'')$ that is $\mathcal{T}$-equivalent to $(A,B)$ and displayed by $\Phi$ with $A''\subseteq C$. 
\end{lemma}

\begin{proof}
Suppose that $(A',B')$ is a $(k,\mathcal{S})$-separation that is $\mathcal{T}$-equivalent to $(A,B)$ with $\fcl_{\mathcal{T}}(A)=\fcl_{\mathcal{T}}(A')$, and that $A'\subseteq C$. We may assume, up to labels, that $C=P_1\cup \dots \cup P_j$ for some $j\in [n-1]$. Let $I\subseteq [n]$ be the set of indices such that $P_I=A$, and let $J=[j]$. Since $(A,B)$ and $(A',B')$ are, in particular, non-sequential $k$-separations and $\fcl_{\mathcal{T}}(A)=\fcl_{\mathcal{T}}(A')$, it follows that $A$ meets $A'$. Thus $K=I\cap J\subseteq [n]$ is a non-empty set of indices. Moreover, it is easily seen that $P_K$ is a $k$-separating set with $P_K\subseteq C$. Now from the fact that $E-\fcl_{\mathcal{T}}(B)\subseteq A\cap A'\subseteq P_K\subseteq A$, it follows that $(P_K, E-P_K)$ is a $k$-separation displayed by $\Phi$ that is $\mathcal{T}$-equivalent to $(A,B)$ Lemma \ref{bigeqlemma} (iv), and thus $(P_K, E-P_K)$ is a $(k,\mathcal{S})$-separation by (S1). 
\end{proof}

\subsection*{Flowers in vertically $k$-connected matroids}
Let $M$ be a vertically $k$-connected matroids where $r(M)\geq \max\{3k-5,2\}$. Then a 
flower relative to $\tangle_k$ is a partition $(P_1,\ldots,P_n)$ such that $r(P_i)\geq k-1$,
and both $P_i$ and $P_i\cup P_{i+1}$ are $k$-separating for all $i$. In particular,
if $M$ is strictly $k$-connected we can replace the condition $r(P_i)\geq k-1$ by the
condition that $|P_i|\geq k-1$. Via these interpretations, all of the lemmas of this 
section have straightforward specialisations for flowers in vertically $k$-connected
matroids.

\section{Conformity}
\label{conf}

Let $\mathcal{T}$ be a tangle of order $k$ in a connectivity system $(E,\lambda)$, and let $\mathcal{S}$ be a tree compatible set. A $k$-flower in $\mathcal{T}$ is $\mathcal{S}$-\textit{maximal} if it is maximal in the quasi-order $\preccurlyeq_{\mathcal{S}}$. The main goal of this section is to prove Theorem \ref{TMkflowerconf}, which, loosely stated, is to show that if $\mathcal{T}$ is a tangle of order $k$ that satisfies a certain robustness condition, then every $(k,\mathcal{S})$-separation ``conforms'' with an $\mathcal{S}$-tight $\mathcal{S}$-maximal flower in $\mathcal{T}$. We first study how the $(k,\mathcal{S})$-separations interact with $k$-flowers in $\mathcal{T}$, and we develop the necessary lemmas to prove Theorem \ref{TMkflowerconf}.

To avoid cumbersome statements we assume for the remainder of this section that $\mathcal{T}$ is a tangle of order $k$ in a connectivity system $(E,\lambda)$, and that $\mathcal{S}$ is a tree compatible set. 

We begin with the following easy lemma about certain subsets of petals of a $k$-flower $\Phi$ in $\mathcal{T}$.

\begin{lemma}
\label{strpetal}
Let $n\geq 2$, and let $\Phi=(P_1,\ldots,P_n)$ be a $k$-flower in $\mathcal{T}$. Let $I$ be a proper non-empty subset of $[n]$. 
\begin{enumerate}
 \item[(i)] If $X\subseteq P_I$ is a $\mathcal{T}$-strong set, then $\lambda(X)\geq k$.
 \item[(ii)] If $X\subseteq P_I$ and $\lambda(X)<k$, then $X\in \mathcal{T}$. 
\end{enumerate}
\end{lemma}  

\begin{proof}
For (i), suppose that $X\subseteq P_I$ is a $\mathcal{T}$-strong set. Then there is some $i\in [n]-I$, so $P_i\subseteq E-X$. Thus $(X,E-X)$ is a $\mathcal{T}$-strong partition, so $\lambda(X)\geq k$ by (T2).

For (ii), suppose that $X\subseteq P_I$ and $\lambda(X)<k$. Then $X$ or $E-X$ belongs to $\mathcal{T}$ by (T2), and $E-X$ is $\mathcal{T}$-strong, so $X\in \mathcal{T}$. 
\end{proof}

Let $n\geq 2$, and let $\Phi=(P_1,\ldots, P_n)$ be a $k$-flower in $\mathcal{T}$. Let $I\subseteq [n]$ be a proper non-empty set. A $k$-separation $(R,G)$ of $\lambda$ \textit{crosses} a union of petals $P_I$ of $\Phi$ if both $P_I\cap R$ and $P_I\cap G$ are non-empty sets. We say that $P_I$ is \textit{strongly crossed} by $(R,G)$ if both $P_I\cap R$ and $P_I\cap G$ are $\mathcal{T}$-strong sets, and that $P_I$ is \textit{weakly crossed} by $(R,G)$ if both $P_I\cap R$ and $P_I\cap G$ are $\mathcal{T}$-weak sets. 

A $\mathcal{T}$-strong $k$-separation $(R,G)$ is said to \textit{conform} with a $k$-flower $\Phi$ in $\mathcal{T}$ if either $(R,G)$ is $\mathcal{T}$-equivalent to a $k$-separation that is displayed by $\Phi$ or $(R,G)$ is $\mathcal{T}$-equivalent to a $k$-separation $(R',G')$ with the property that $R'$ or $G'$ is contained in a petal of $\Phi$.

Let $\Phi$ be a $k$-flower in $\mathcal{T}$, and let $(R,G)$ be a $(k,\mathcal{S})$-separation that does not conform with $\Phi$. Then it is easy to see that $(R,G)$ crosses some petal of $\Phi$. We would like show that there is a $k$-flower that both refines $\Phi$ and displays a $(k,\mathcal{S})$-separation that is $\mathcal{T}$-equivalent to $(R,G)$. 

A $\mathcal{T}$-strong $k$-separation $(R,G)$ called \textit{$\Phi$-minimum} if, among the $k$-separations that are $\mathcal{T}$-equivalent to $(R,G)$, it crosses a minimum number of petals of $\Phi$.

\begin{lemma}
\label{cross2}
Let $n\geq 2$, and let $\Phi=(P_1,\ldots, P_n)$ be a $k$-flower in $\mathcal{T}$. Let $I$ be a proper non-empty subset of $[n]$ such that $P_I$ is a $k$-separating, and let $(R,G)$ be a $\Phi$-minimum $(k,\mathcal{S})$-separation that crosses $P_I$. 
\begin{enumerate}
 \item[(i)] If $\lambda(P_I\cap R)\geq k$, then $P_I\cap G$ is $\mathcal{T}$-strong. 
 \item[(ii)] If $P_I$ is weakly crossed by $(R,G)$, then $P_I\cap G$ and $P_I\cap R$ are both members of $\mathcal{T}$. 
 \item[(iii)] If $P_I$ is weakly crossed by $(R,G)$, then $P_I$ is a sequential $k$-separating set.
\end{enumerate}
\end{lemma}

\begin{proof}
For (i), assume that $\lambda(P_I\cap R)\geq k$. Seeking a contradiction, suppose that $P_I\cap G$ is $\mathcal{T}$-weak. Then $P_I\cup R$ is $k$-separating by uncrossing $P_I$ and $R$, so $(R,G)$ is $\mathcal{T}$-equivalent to $(R\cup P_I,G-P_I)$ by Lemma \ref{bigeqlemma}. But $(R\cup P_I,G-P_I)$ crosses fewer petals of $\Phi$ than $(R,G)$; a contradiction because $(R,G)$ is $\Phi$-minimum.

For (ii), we first show that $\lambda(P_I\cap R)<k$ and $\lambda(P_I\cap G)<k$. Seeking a contradiction, assume, up to switching $R$ and $G$, that $\lambda(P_I\cap R)\geq k$. Then $P_I\cap G$ is $\mathcal{T}$-strong by (i); a contradiction because $P_I$ is weakly crossed by $(R,G)$. Thus $\lambda(P_I\cap R)<k$ and $\lambda(P_I\cap G)<k$. It now follows from Lemma \ref{strpetal} (ii) that $P_I\cap G, P_I\cap R\in \mathcal{T}$. 

For (iii), observe that $E-P_I$ is a $\mathcal{T}$-strong $k$-separating set because $I$ is a proper non-empty subset of $[n]$ such that $P_I$ is $k$-separating. Then $P_I\cap G$ and $P_I\cap R$ are $\mathcal{T}$-weak sets by (ii), and moreover $\lambda((E-P_I)\cup (P_I\cap G))=\lambda(P_I\cap R)<k$ by (ii), so $(P_I\cap G, P_I\cap R)$ is a partial $k$-sequence for $E-P_I$. Thus $P_I\subseteq \fcl_{\mathcal{T}}(E-P_I)$ by Lemma \ref{fclcontain}, and so $\fcl_{\mathcal{T}}(E-P_I)=E$.
\end{proof}

The next lemma shows that a $k$-separating proper non-empty union of petals of $\Phi$ is either strongly or weakly crossed by a $\Phi$-minimum $(k,\mathcal{S})$-separation $(R,G)$.

\begin{lemma}
\label{cross4}
Let $n\geq 2$, and let $\Phi=(P_1,\ldots, P_n)$ be a $k$-flower in $\mathcal{T}$. Let $I$ be a proper non-empty subset of $[n]$ such that $P_I$ is $k$-separating. If $(R,G)$ is a $\Phi$-minimum $(k,\mathcal{S})$-separation that crosses $P_I$, then $P_I$ is either strongly or weakly crossed. 
\end{lemma}

\begin{proof}
Assume that $(R,G)$ is a $\Phi$-minimum $(k,\mathcal{S})$-separation that crosses $P_I$. If $\lambda(P_I\cap R)<k$ and $\lambda(P_I\cap G)<k$, then $(R,G)$ weakly crosses $P_I$ by Lemma \ref{strpetal} (ii). Thus we may assume, up to switching $R$ and $G$, that $\lambda(P_I\cap R)\geq k$. Then $P_I\cap G$ is $\mathcal{T}$-strong by Lemma \ref{cross2} (i). Thus $\lambda(P_I\cap G)\geq k$ by Lemma \ref{strpetal} (i). Then $P_I\cap R$ is also $\mathcal{T}$-strong by Lemma \ref{cross2}(i). Therefore $(R,G)$ strongly crosses $P_I$. 
\end{proof}

The next lemma shows that if $(R,G)$ is a $(k,\mathcal{S})$-separation does not conform with an $\mathcal{S}$-tight $k$-flower $\Phi$ in $\mathcal{T}$ of $\mathcal{S}$-order two, then there is a $k$-flower in $\mathcal{T}$ that refines $\Phi$ and displays a $(k,\mathcal{S})$-separation that is $\mathcal{T}$-equivalent to $(R,G)$.

\begin{lemma}
\label{ref1}
 Let $\Phi=(P_1,P_2)$ be an $\mathcal{S}$-tight $k$-flower in $\mathcal{T}$. If $(R,G)$ is a $(k,\mathcal{S})$-separation that does not conform with $\Phi$, then there is a $k$-flower $\Phi '$ that refines $\Phi$ and displays a $(k,\mathcal{S})$-separation that is $\mathcal{T}$-equivalent to $(R,G)$.
\end{lemma}
 
\begin{proof}
Suppose that $(R,G)$ is a $(k,\mathcal{S})$-separation that does not conform with $\Phi$. We may assume, by possibly replacing $(R,G)$ by a $\mathcal{T}$-equivalent $(k,\mathcal{S})$-separation, that $(R,G)$ is $\Phi$-minimum. Clearly both $P_1$ and $P_2$ are crossed by $(R,G)$ because it does not conform with $\Phi$. We claim that $\Phi '=(P_1\cap G,P_1\cap R,P_2\cap R, P_2\cap G)$ is a $k$-flower in $\mathcal{T}$. Since $\Phi$ is $\mathcal{S}$-tight, it follows that $(P_1,P_2)$ is a $(k,\mathcal{S})$-separation, so by Lemma \ref{cross2} (iii) and the fact that $(P_1,P_2)$ is non-sequential both $P_1$ and $P_2$ are strongly crossed by $(R,G)$. Thus $\Phi '$ is a $\mathcal{T}$-strong partition. Furthermore, the union of any two consecutive petals of $\Phi'$ is a member of $\{R, G, P_1, P_2\}$, and so $k$-separating. Thus $\Phi'$ a $k$-flower in $\mathcal{T}$ by Lemma \ref{flowerunion}, and $\Phi \preccurlyeq_{\mathcal{S}}\Phi'$.
\end{proof}

Let $n\geq 2$, and let $\Phi=(P_1,\ldots, P_n)$ be a $k$-flower in $\mathcal{T}$. Suppose that $(R,G)$ is a $\Phi$-minimum $(k,\mathcal{S})$-separation that does not conform with $\Phi$. Let $I$ be a proper non-empty subset of $[n]$ such that $P_I$ is $k$-separating. We say that $P_I$ is \textit{$(R,G)$-strong} if either $P_I$ is not crossed by $(R,G)$ or $P_I$ is strongly crossed by $(R,G)$, and that $P_I$ is \textit{$(R,G)$-weak} if $P_I$ is weakly crossed by $(R,G)$. By Lemma \ref{cross4}, $P_I$ is either $(R,G)$-weak or $(R,G)$-strong. Evidently, if a petal $P_i$ of $\Phi$ is $(R,G)$-strong, then $P_i\cap R$ or $P_i\cap G$ is $\mathcal{T}$-strong. 

The next lemma shows that $(R,G)$-weak petals of $\Phi$ are the only obstacles to finding a $k$-flower that refines $\Phi$ and displays a $(k,\mathcal{S})$-separation that is $\mathcal{T}$-equivalent to $(R,G)$.

\begin{lemma}
 \label{ref2}
 Let $n\geq 3$, and let $\Phi=(P_1,P_2,\ldots,P_n)$ be a $k$-flower in $\mathcal{T}$. Let $(R,G)$ be a $\Phi$-minimum $(k,\mathcal{S})$-separation that does not conform with $\Phi$. If every petal of $\Phi$ is $(R,G)$-strong, then there is a $k$-flower that refines $\Phi$ and displays $(R,G)$.
\end{lemma}

\begin{proof}
Suppose that every petal of $\Phi$ is $(R,G)$-strong. Then, up to labels, we may assume that $(R,G)$ crosses $P_1$. Let $P_3'=P_3\cup \cdots \cup P_n$.

\begin{claim}
 Up to switching $R$ and $G$, both $P_2\cap R$ and $P_3'\cap G$ are $\mathcal{T}$-strong.
\end{claim}
 
\begin{subproof}
  If $(R,G)$ crosses $P_2$, then both $P_2\cap R$ and $P_2\cap G$ are $\mathcal{T}$-strong. Up to switching $R$ and $G$, we may assume that $P_3'\cap G$ is $\mathcal{T}$-strong, so both $P_2\cap R$ and $P_3'\cap G$ are $\mathcal{T}$-strong. Thus we may assume that $(R,G)$ does not cross $P_2$. Then, up to switching $R$ and $G$, we can assume that $P_2\subseteq R$, so $P_2\cap R$ is $\mathcal{T}$-strong. Now if $G$ avoids $P_3'$, then $G\subseteq P_1$; a contradiction because $(R,G)$ does not conform with $\Phi$. Thus $G$ meets $P_3'$, so $P_i\cap G$ is $\mathcal{T}$-strong for some $i\in [3,n]$. Hence $P_3'\cap G$ is $\mathcal{T}$-strong. 
\end{subproof}

Assume that labels are chosen such that $P_2\cap R$ and $P_3'\cap G$ are $\mathcal{T}$-strong.

\begin{claim}
\label{ref2.2}
 $\Phi'=(P_1\cap G, P_1\cap R, P_2,\ldots, P_n)$ is a $k$-flower in $\mathcal{T}$.
\end{claim}

\begin{subproof}
  The members of the partition $\Phi '$ are $\mathcal{T}$-strong. Furthermore, $(P_1\cap G)\cup (P_1\cap R)$ is $k$-separating, and $P_i\cup P_{i+1}$ is $k$-separating for all $i\in [2,n-1]$. Thus, by Lemma \ref{flowerunion}, it suffices to show that $(P_1\cap R)\cup P_2$ is $k$-separating. Now $\lambda(P_1\cup P_2\cup R)=\lambda(P_3'\cap G)\geq k$ by Lemma \ref{strpetal} (i), so, by uncrossing $P_1\cup P_2$ and $R$, we see that $(P_1\cup P_2)\cap R$ is $k$-separating. Furthermore $\lambda(P_2\cap R)\geq k$ by Lemma \ref{strpetal} (i), so $(P_1\cap R)\cup P_2$ is $k$-separating by uncrossing $P_2$ and $(P_1\cup P_2)\cap R$. 
\end{subproof}

It now follows from \ref{ref2.2} and an induction on the number of petals of $\Phi$ crossed by $(R,G)$ that there is a $k$-flower $\Phi '$ that refines $\Phi$ and displays $(R,G)$.
\end{proof}

For $\mathcal{S}$-tight $k$-flowers we only need two $(R,G)$-strong petals to guarantee that every petal is $(R,G)$-strong. To show this we first need the following lemma.

\begin{lemma}
 \label{phiparti}
Let $n\geq 3$, and let $\Phi=(P_1,\ldots, P_n)$ be a $k$-flower in $\mathcal{T}$. Let $(R,G)$ be a $\Phi$-minimum $(k,\mathcal{S})$-separation that does not conform with $\Phi$. If $P_1$ is $(R,G)$-weak and there is a concatenation $(P_1,A,B)$ of $\Phi$ such that both $A$ and $B$ are $(R,G)$-strong, then $\Phi$ is equivalent to the $k$-flower $\Phi' =(P_1\cup P_2, P_3,\ldots, P_n)$. 
\end{lemma}

\begin{proof}
Suppose that $P_1$ is $(R,G)$-weak, and that $(P_1,A,B)$ is a concatenation of $\Phi$ such that both $A$ and $B$ are $(R,G)$-strong. 

\begin{claim}
 Up to switching $R$ and $G$, both $A\cap R$ and $B\cap G$ are $\mathcal{T}$-strong. 
\end{claim}
 
\begin{subproof}
  Assume first that $(R,G)$ crosses $A$, so both $A\cap R$ and $A\cap G$ are $\mathcal{T}$-strong. Then, up to switching $R$ and $G$, we may assume that $B\cap G$ is $\mathcal{T}$-strong. Thus both $A\cap R$ and $B\cap G$ are $\mathcal{T}$-strong. Now assume that $(R,G)$ does not cross $A$. Then, up to switching $R$ and $G$, we may assume that $A\subseteq R$, so $A\cap R$ is $\mathcal{T}$-strong. If $B\subseteq R$, then $G\subseteq P_1$; a contradiction because $(R,G)$ does not conform with $\Phi$. Thus either $(R,G)$ crosses $B$ or $B\subseteq G$, so $B\cap G$ is $\mathcal{T}$-strong.
\end{subproof}

Assume that $R$ and $G$ are labelled such that both $A\cap R$ and $B\cap G$ are $\mathcal{T}$-strong. Then $\lambda(A\cap R)\geq k$ and $\lambda(B\cap G)\geq k$ by Lemma \ref{strpetal} (i). Since $P_1\cup A\cup R$ avoids $B\cap G$, the set $(P_1\cup A)\cap R$ is $k$-separating by uncrossing $P_1\cup A$ and $R$. Another uncrossing argument with $(P_1\cup A)\cap R$ and $A$, whose intersection is $A\cap R$, shows that their union $A\cup (P_1\cap R)$ is $k$-separating. Finally, $P_2\cup (P_1\cap R)$ is the intersection of the $k$-separating sets $P_1\cup P_2$ and $A\cup (P_1\cap R)$, whose union is $A\cup P_1$, so $P_2\cup (P_1\cap R)$ is $k$-separating by uncrossing $P_1\cup P_2$ and $A\cup (P_1\cap R)$. Then $(P_1\cap R, P_1\cap G)$ is a partial $k$-sequence for $P_2$, so $P_1\subseteq \fcl_{\mathcal{T}}(P_2)$ by Lemma \ref{fclcontain}. Thus $\Phi$ is equivalent to the $k$-flower $(P_1\cup P_2,P_3,\ldots,P_n)$ by Lemma \ref{flowereq}.
\end{proof}

\begin{lemma}
 \label{tight1}
Let $n\geq 3$, and let $\Phi=(P_1,\ldots,P_n)$ be an $\mathcal{S}$-tight $k$-flower in $\mathcal{T}$. Let $(R,G)$ be a $\Phi$-minimal $(k,\mathcal{S})$-separation that does not conform with $\Phi$. If $\Phi$ has two $(R,G)$-strong petals, then every petal of $\Phi$ is $(R,G)$-strong. 
\end{lemma}
 
\begin{proof}
Assume that $\Phi$ has two $(R,G)$-strong petals. Seeking a contradiction, suppose that $\Phi$ has an $(R,G)$-weak petal. Then, up to labels, we may assume that $P_1$ is $(R,G)$-weak, and that $P_2$ is $(R,G)$-strong. Then $P_j$ is $(R,G)$-strong for some $j\notin [1,2]$, so $E-(P_1\cup P_2)$ is $(R,G)$-strong. Now $(P_1,P_2,E-(P_1\cup P_2))$ is a concatenation of $\Phi$ such that $P_1$ is $(R,G)$-weak, and both $P_2$ and $E-(P_1\cup P_2)$ are $(R,G)$-strong. Thus $\Phi$ is equivalent to the $k$-flower $\Phi '=(P_1\cup P_2,P_3,\ldots, P_n)$ by Lemma \ref{phiparti}; a contradiction because $\Phi$ is $\mathcal{S}$-tight. Thus every petal of $\Phi$ is $(R,G)$-strong.
\end{proof}

As a consequence of Lemma \ref{tight1} and Lemma \ref{cross2}(iii) we can refine $\mathcal{S}$-tight $k$-flowers of $\mathcal{S}$-order 3.

\begin{lemma}
 \label{ref3pet}
Let $\Phi=(P_1,P_2,P_3)$ be an $\mathcal{S}$-tight $k$-flower in $\mathcal{T}$. If $(R,G)$ is a $(k,\mathcal{S})$-separation that does not conform with $\Phi$, then there is a $k$-flower that refines $\Phi$ and displays a $(k, \mathcal{S})$-separation that is $\mathcal{T}$-equivalent to $(R,G)$.
\end{lemma}

\begin{proof}
Assume that $(R,G)$ is a $\Phi$-minimum $(k,\mathcal{S})$-separation that does not conform with $\Phi$. As $\Phi$ is an $\mathcal{S}$-tight $k$-flower in $\mathcal{T}$ of $\mathcal{S}$-order three, it displays at least two inequivalent $(k,\mathcal{S})$-separations. By Lemma \ref{cross2} (iii) $(k,\mathcal{S})$-separations displayed by $\Phi$ are strongly crossed by $(R,G)$, so we may assume that $P_1$ and $P_2$ are $(R,G)$-strong. Then all petals of $\Phi$ are $(R,G)$-strong by Lemma \ref{tight1}. It follows from Lemma \ref{ref2} that there is a $k$-flower that refines $\Phi$ and displays a $(k,\mathcal{S})$-separation that is $\mathcal{T}$-equivalent to $(R,G)$.  
\end{proof}

We can do one better than Lemma \ref{tight1}.

\begin{lemma}
 \label{1strongpetal}
Let $n\geq 3$, and let $\Phi=(P_1,\ldots, P_n)$ be an $\mathcal{S}$-tight $k$-flower in $\mathcal{T}$. Let $(R,G)$ be a $\Phi$-minimum $(k,\mathcal{S})$-separation that does not conform with $\Phi$. If $\Phi$ has one $(R,G)$-strong petal, then every petal of $\Phi$ is $(R,G)$-strong. 
\end{lemma}

\begin{proof}
If $\Phi$ has two $(R,G)$-strong petals, then the conclusion follows from Lemma \ref{tight1}. Assume towards a contradiction that $\Phi$ has exactly one $(R,G)$-strong petal. Then, up to labels, we can assume that $P_2$ is $(R,G)$-strong. Since $\Phi$ is an $\mathcal{S}$-tight $k$-flower of $\mathcal{S}$-order at least three, it displays some $(k,\mathcal{S})$-separation $(X,Y)$ that is not $\mathcal{T}$-equivalent to $(P_2,E-P_2)$. By Lemma \ref{cross2} (iii) both $X$ and $Y$ are $(R,G)$-strong, so they must contain at least two petals of $\Phi$. Then we can assume that $(X,Y)$ and the petals of $\Phi$ are labelled such that $P_2$ is $(R,G)$-strong and $P_1,P_2\subseteq X$. Now $Y\subseteq E-(P_1\cup P_2)$, so $E-(P_1\cup P_2)$ is also $(R,G)$-strong. Thus $(P_1,P_2,E-(P_1\cup P_2))$ is a concatenation of $\Phi$ such that $P_1$ is $(R,G)$-weak, and both $P_2$ and $E-(P_1\cup P_2)$ are $(R,G)$-strong. Then it follows from Lemma \ref{phiparti} that $\Phi '=(P_1\cup P_2,P_3,\ldots, P_n)$ is a $k$-flower in $\mathcal{T}$ that is equivalent to $\Phi$; a contradiction because $\Phi$ is $\mathcal{S}$-tight. Thus $\Phi$ has two $(R,G)$-strong petals, so by Lemma \ref{tight1} every petal of $\Phi$ is $(R,G)$-strong.
\end{proof}

Unfortunately, Lemma \ref{1strongpetal} is as much as we can say for arbitrary tangles in a connectivity system. Consider the $8$-element rank-$4$ matroid $R_8$ that is represented geometrically by a cube (see, for example, \cite[pp. 646]{oxley2011matroid}). The $4$-point planes of $R_8$ are the six faces of the cube and the six diagonal planes. Let $E=[8]$ be the ground set of $R_8$, and let $r$ be the rank function of $R_8$. For each positive integer $\ell$, define a function $f_{\ell}$ on the subsets $X$ of $E$ by

\[ f_{\ell}(X)=\begin{cases}
 0 & X=\emptyset. \\
 r(X)+\ell & \text{ otherwise. }
\end{cases} \]

It is straightforward to prove that $f_{\ell}$ is a polymatroid on $E=[8]$. Let $\lambda_{\ell}$ be the connectivity function of $f_{\ell}$, that is, $\lambda_{\ell}(X)=f_{\ell}(X)+f_{\ell}(E-X)-f_{\ell}(E)+1$ for all $X\subseteq E$. Then $\mathcal{T}=\{\{i\}\ |\ i\in [8]\}\cup \{\emptyset\}$ is the unique tangle in $(E, \lambda_{\ell})$ of order $\ell+3$. Let $\mathcal{S}$ be the set of all non-sequential $(\ell+3)$-separating sets in $\lambda_{\ell}$ with $\mathcal{T}$-strong complements. With notation as in Figure \ref{R82}, the partition $\Phi= (\{1,2\},\{3,4\},\{5,6\},\{7,8\})$ is an $\mathcal{S}$-tight $\mathcal{S}$-maximal $(\ell+3)$-flower in $\mathcal{T}$. However, the non-sequential $(\ell+3)$-separation $(\{1,3,5,7\},\{2,4,6,8\})$ does not conform with $\Phi$.

\begin{figure}[hbt]
 \begin{center}
 \begin{tikzpicture}[thick,line join=round]
                \coordinate[label=180:$1$] (a) at (0,0);
                \coordinate[label=180:$5$] (b) at (0,2.2);
                \coordinate[label=0:$2$] (c) at (2,0);
                \coordinate[label=-25:$6$] (d) at (2,2.2);
                \coordinate[label=180:$4$] (e) at ($(a) + (50:1.3)$);
                \coordinate[label=180:$8$] (f) at ($(b) + (50:1.3)$);
                \coordinate[label=0:$3$] (g) at ($(c) + (50:1.3)$);
                \coordinate[label=0:$7$] (h) at ($(d) + (50:1.3)$);
                \draw (a) -- (b) -- (d) -- (c) -- (g) -- (h) -- (f) -- (e) -- (a) -- (c);
                \draw (b) -- (f);
                \draw (d) -- (h);
                \draw (e) -- (g);
                \foreach \pt in {a,b,c,d,e,f,g,h} \fill[black] (\pt) circle (3pt);
        \end{tikzpicture}

\end{center}
\caption{The matroid, $R_8$.}\label{R82} 
\end{figure} 

We can obtain analogous matroid examples by a standard construction where matroid elements are freely added to each polymatroid element.

Let $k$ be a positive integer, and let $\mathcal{T}$ be a collection of subsets of $E$ satisfying the axioms (T1), (T2), and (T4). Then $\mathcal{T}$ is a \textit{robust tangle of order $k$} in a connectivity system $(E,\lambda)$ if the following property holds:
\begin{enumerate}
 \item[(RT3)] If $A_1,A_2,\ldots,A_8\in \mathcal{T}$, then $A_1\cup A_2\cup \cdots \cup A_8\neq E$
\end{enumerate}

Note that every robust tangle of order $k$ in $(E,\lambda)$ is certainly a tangle of order $k$ in $(E,\lambda)$. 

We are now in position to achieve the main goal of this section.

\begin{thm}
 \label{TMkflowerconf}
Let $\mathcal{T}$ be a robust tangle of order $k$ in a connectivity system $(E,\lambda)$, and let $\mathcal{S}$ be a tree compatible set. If $\Phi$ is an $\mathcal{S}$-tight $\mathcal{S}$-maximal $k$-flower in $\mathcal{T}$, then every $(k,\mathcal{S})$-separation conforms with $\Phi$. 
\end{thm}

\begin{proof}
Let $\Phi=(P_1,P_2,\ldots, P_n)$ be an $\mathcal{S}$-tight $\mathcal{S}$-maximal $k$-flower in $\mathcal{T}$. Assume that the theorem fails, and that $(R,G)$ is a $\Phi$-minimum $(k,\mathcal{S})$-separation that does not conform with $\Phi$. Then clearly $n\geq 2$, and by Lemma \ref{ref1} and Lemma \ref{ref3pet}, we may assume that $n\geq 4$. Assume towards a contradiction that every petal of $\Phi$ is $(R,G)$-weak. For each $i\in [n-3]$, let $A_i=P_2\cup P_3\cup \cdots \cup P_{n-i-1}$ and let $B_i=P_n\cup \cdots \cup P_{n-i+1}$, and consider the concatenation $\Phi_{i}=(P_1,A_i,P_{n-i},B_i)$ of $\Phi$. The petals of $\Phi_{i}$ are either $(R,G)$-weak or $(R,G)$-strong by Lemma \ref{cross4}, and they cannot all be $(R,G)$-weak because $\mathcal{T}$ satisfies (RT3). For each $i\in [n-3]$, the petals $P_1$ and $P_{n-i}$ are $(R,G)$-weak, so $A_i$ or $B_i$ must be $(R,G)$-strong. Moreover, both $A_1$ and $B_{n-3}$ are $(R,G)$-strong because both $B_1=P_n$ and $A_{n-3}=P_2$ are $(R,G)$-weak. Thus there is a smallest index $j\geq 2$ such that $B_j$ is $(R,G)$-strong. Then $B_{j-1}$ is $(R,G)$-weak, so $A_{j-1}$ is $(R,G)$-strong by (RT3). Now $(P_1,A_{j-1},B_j)$ is a concatenation of $\Phi$ such that both $A_{j-1}$ and $B_j$ are $(R,G)$-strong, so it follows from Lemma \ref{phiparti} that $\Phi '=(P_1\cup P_2,P_3,\ldots, P_n)$ is a $k$-flower equivalent to $\Phi$; a contradiction because $\Phi$ is $\mathcal{S}$-tight. Thus $\Phi$ has an $(R,G)$-strong petal. It now follows from Lemma \ref{1strongpetal} that every petal of $\Phi$ is $(R,G)$-strong. Then, by Lemma \ref{ref2}, there is a $k$-flower that refines $\Phi$ and displays a $(k,\mathcal{S})$-separation that is $\mathcal{T}$-equivalent to $(R,G)$; a contradiction of the $\mathcal{S}$-maximality of $\Phi$.
\end{proof}

\section{Partial k-trees}
\label{ktree}

The tree used to obtain the tree decomposition of $3$-connected matroids in \cite{oxley2004structure} was a $\pi$-labelled tree called a maximal partial $3$-tree. We will use an analogous $\pi$-labelled tree to obtain the tree decomposition in Theorem \ref{bigone}. The exposition given here will therefore closely follow that of Oxley, Semple, and Whittle \cite{oxley2004structure}.

Let $\mathcal{T}$ be a tangle of order $k$ in a connectivity system $(E,\lambda)$, and let $\pi$ be a partition of $E$. Note that we allow members of $\pi$ to be empty. Let $T$ be a tree such that every member of $\pi$ labels a vertex of $T$. Some vertices may be unlabelled and no vertex is multiply labelled. We say that $T$ is a \textit{$\pi$-labelled tree} for $\mathcal{T}$. The vertices of $T$ labelled by the members of $\pi$ are called \textit{bag vertices}, and the members of $\pi$ are called \textit{bags}. A \textit{terminal bag} is a bag that labels a leaf of $T$.

Let $T$ be a $\pi$-labelled tree for $\mathcal{T}$. We now define some partitions of $E$ that are induced by certain subgraphs of $T$. Let $T'$ be a subtree of $T$. The union of those bags that label vertices of $T'$ is the subset of $E$ \textit{displayed} by $T'$. Let $e$ be an edge of $T$. The \textit{partition of $E$ displayed by $e$} is the partition displayed by the connected components of $T\del e$. Let $v$ be a vertex of $T$ that is not a bag vertex. Then the \textit{partition of $E$ displayed by $v$} is the partition displayed by the connected components of $T-v$. The edges incident with $v$ are in natural one-to-one correspondence with the connected components of $T-v$, and hence with the members of the partition of $E$ displayed by $v$. In what follows, if a cyclic ordering is imposed on the edges incident with $v$, then we cyclically order the members of the partition of $E$ displayed by $v$ in the corresponding order. 

Let $v$ be a vertex of a $\pi$-labelled tree $T$ that is not a bag vertex, and let $(e_1,\ldots,e_n)$ be a cyclic ordering of the edges incident with $v$. Then $v$ is a \textit{$k$-flower vertex} if the partition $(P_1,\ldots,P_n)$ of $E$ displayed by $v$, in the cyclic order corresponding to $(e_1,\ldots,e_n)$, is a $k$-flower in $\mathcal{T}$. The $k$-separations displayed by the $k$-flower corresponding to a $k$-flower vertex are called the $k$-separations \textit{displayed by $v$}. A $k$-separation is \textit{displayed by $T$} if it is displayed by some edge or some $k$-flower vertex of $T$. A $\mathcal{T}$-strong $k$-separation $(X,Y)$ \textit{conforms} with $T$ if either $(X,Y)$ is $\mathcal{T}$-equivalent to a $k$-separation displayed by $T$, or $(X,Y)$ is $\mathcal{T}$-equivalent to a $k$-separation $(X',Y')$ with the property that $X'$ or $Y'$ is contained in a bag of $T$. 

Let $\mathcal{T}$ be a tangle of order $k$ in a connectivity system $(E,\lambda)$, and let $\mathcal{S}$ be a tree compatible set. A \textit{partial $(k,\mathcal{S})$-tree} for $\mathcal{T}$ is a $\pi$-labelled tree for $\mathcal{T}$, where $\pi$ is a partition of $E$ such that the following properties hold:
\begin{enumerate}
 \item[(P1)] For each edge $e$ of $T$, the partition $(X,Y)$ of $E$ displayed by $e$ is a $\mathcal{T}$-strong $k$-separation, and, if $e$ is incident with two bag vertices, then $(X,Y)$ is a $(k,\mathcal{S})$-separation.
 \item[(P2)] Each non-bag vertex $v$ of $T$ is labelled either $D$ or $A$. Moreover, if $v$ is labelled by $D$, then there is a cyclic ordering on the edges incident with $v$.
 \item[(P3)] If a vertex $v$ if labelled by $A$, then the partition of $E$ displayed by $v$ is a $k$-anemone of $\mathcal{S}$-order at least three with no $\mathcal{T}$-loose petals.
 \item[(P4)] If a vertex $v$ is labelled by $D$, then the partition of $E$ displayed by $v$, in the cyclic order induced by the cyclic ordering on the edges incident with $v$, is a $k$-daisy of $\mathcal{S}$-order at least three with no $\mathcal{T}$-loose petals.
 \item[(P5)] Every $(k,\mathcal{S})$-separation conforms with $T$.
\end{enumerate} 

Note that if $(X,Y)$ is displayed by an edge $e$ of a partial $(k,\mathcal{S})$-tree, then $X\in \mathcal{S}$ or $Y\in \mathcal{S}$. This follows from (P1) if $e$ is incident with two bag vertices, and from Lemma \ref{nompetalt} if $e$ is incident with a $k$-flower vertex. 

We now define a relation $\preccurlyeq_{\mathcal{S}}$ on the set of partial $(k,\mathcal{S})$-trees for $\mathcal{T}$. Let $T$ and $T'$ be partial $(k,\mathcal{S})$-trees for $\mathcal{T}$. If, for each $(k,\mathcal{S})$-separation displayed by $T$, there is some $\mathcal{T}$-equivalent $(k,\mathcal{S})$-separation displayed by $T'$, then $T\preccurlyeq_{\mathcal{S}} T'$. It is straightforward to check that $\preccurlyeq_{\mathcal{S}}$ is a quasi-order on the set of partial $(k,\mathcal{S})$-trees for $\mathcal{T}$. If $T\preccurlyeq_{\mathcal{S}} T'$ and $T'\preccurlyeq_{\mathcal{S}} T$, then $T$ is \textit{$\mathcal{T}$-equivalent} to $T'$ with respect to $\mathcal{S}$. As with the other notions of equivalence we have developed, when the tangle $\mathcal{T}$ and the set $\mathcal{S}$ are clear from the context, we shall abbreviate ``$\mathcal{T}$-equivalent with respect to $\mathcal{S}$'' to ``equivalent''. A partial $(k,\mathcal{S})$-tree is $\mathcal{S}$-\textit{maximal} if it is maximal in the quasi-order $\preccurlyeq_{\mathcal{S}}$. A partial $(k,\mathcal{S})$-tree for $\mathcal{T}$ is \textit{trivial} if it does not display any $(k,\mathcal{S})$-separations. 

Let $\Phi=(P_1,\ldots, P_n)$ be a $k$-flower in $\mathcal{T}$, and let $\mathcal{S}$ be a tree compatible set. There is a $\Phi$-labelled tree for $\mathcal{T}$ that we can associate with $\Phi$. If $n=1$, then $T$ consists of a single bag-vertex labelled by the bag $P_1$. If $n=2$, then $T$ consists of two adjacent bag vertices labelled by $P_1$ and $P_2$ respectively. Assume that $n\geq 3$. Then we let $T$ be the tree with vertex set $\{v, v_1,\ldots, v_n \}$, where $v$ is adjacent to each $v_i$, and each $v_i$ is labelled by the bag $P_i$. Finally, if $\Phi$ is a $k$-daisy, then the edges incident with the non-bag vertex $v$ are given the cyclic ordering $(vv_1,\ldots,vv_n)$. If $\Phi$ is an $\mathcal{S}$-tight $k$-flower in $\mathcal{T}$, then it is easily seen that the associated $\Phi$-labelled tree for $\mathcal{T}$ satisfies the first four partial $(k,\mathcal{S})$-tree axioms. Moreover, we have the following immediate consequence of Theorem \ref{TMkflowerconf}.

\begin{cor}
 \label{TMkflower}
Let $\mathcal{T}$ be a robust tangle of order $k$ in a connectivity system $(E,\lambda)$, and let $\mathcal{S}$ be a tree compatible set. If $\Phi$ is an $\mathcal{S}$-tight $\mathcal{S}$-maximal $k$-flower in $\mathcal{T}$, then the $\Phi$-labelled tree associated with $\Phi$ is a partial $(k,\mathcal{S})$-tree for $\mathcal{T}$. 
\end{cor}
  
The next result is used in the proof of Theorem \ref{bigone}.

\begin{lemma}
 \label{TMflow}
Let $\mathcal{T}$ be a tangle of order $k$ in a connectivity system $(E,\lambda)$, and let $\mathcal{S}$ be a tree compatible set. If $(R,G)$ is a $(k,\mathcal{S})$-separation, then there is an $\mathcal{S}$-tight $\mathcal{S}$-maximal $k$-flower in $\mathcal{T}$ that displays a $(k,\mathcal{S})$-separation that is $\mathcal{T}$-equivalent to $(R,G)$. 
\end{lemma}

\begin{proof}
Assume that $(R,G)$ is a $(k,\mathcal{S})$-separation of $\lambda$. Then $\Phi=(R,G)$ is an $\mathcal{S}$-tight $k$-flower in $\mathcal{T}$, and evidently $\Phi$ displays $(R,G)$. Let $\Phi '$ be an $\mathcal{S}$-maximal $k$-flower in $\mathcal{T}$ such that $\Phi '\succcurlyeq_{\mathcal{S}} \Phi$, and let $\Phi ''$ be an $\mathcal{S}$-tight $k$-flower in $\mathcal{T}$ that is equivalent to $\Phi '$. Then $\Phi ''$ is an $\mathcal{S}$-tight $\mathcal{S}$-maximal $k$-flower in $\mathcal{T}$, and $\Phi '' \succcurlyeq_{\mathcal{S}}\Phi$, so $\Phi ''$ displays a $(k,\mathcal{S})$-separation $(R',G')$ that is $\mathcal{T}$-equivalent to $(R,G)$.
\end{proof}

The remainder of this section is devoted to developing the preliminary lemmas needed to prove the following lemma, which is the main component in the proof of Theorem \ref{bigone}. 

\begin{lemma}
 \label{tree}
Let $\mathcal{T}$ be a robust tangle of order $k$ in a connectivity system $(E,\lambda)$, and let $\mathcal{S}$ be a tree compatible set. Let $T$ be a non-trivial partial $(k,\mathcal{S})$-tree for $\mathcal{T}$. If there is a $(k,\mathcal{S})$-separation $(R,G)$ that is not $\mathcal{T}$-equivalent to any $(k,\mathcal{S})$-separation displayed by $T$, then there is a partial $(k,\mathcal{S})$-tree $T'$ such that $T'\succcurlyeq_{\mathcal{S}} T$ and $T'$ displays some $(k,\mathcal{S})$-separation that is not displayed by $T$.
\end{lemma}

Let $T$ be a partial $(k,\mathcal{S})$-tree for $\mathcal{T}$. If $B$ is a terminal bag of $T$ such that the partition $(B,E-B)$ is a $(k,\mathcal{S})$-separation, then $B$ is called an \textit{$\mathcal{S}$-terminal-bag} of $T$. The main step towards a proof of Lemma \ref{tree} is to show that if $T$ has an $\mathcal{S}$-terminal-bag $B$, and $(C,E-C)$ is a $(k,\mathcal{S})$-separation such that $\fcl_{\mathcal{T}}(B)=\fcl_{\mathcal{T}}(C)$, then there is some partial $(k,\mathcal{S})$-tree $T'$ that is equivalent to $T$ such that $C$ is a terminal bag of $T'$.

We say that two $k$-separations $(A,B)$ and $(C,D)$ of $\mathcal{T}$ \textit{cross} if the intersections $A\cap C$, $A\cap D$, $B\cap C$, and $B\cap D$ are all non-empty. A set $\mathcal{S}$ of $k$-separations of $\mathcal{T}$ is \textit{laminar} if no two separations in $\mathcal{S}$ cross. We have the following straightforward lemma. We omit the routine proof.

\begin{lemma}
\label{laminar}
If $T$ is a partial $(k,\mathcal{S})$-tree for $\mathcal{T}$, then the set of $k$-separations displayed by edges of $T$ is laminar.
\end{lemma}

\begin{lemma}
\label{bigeqtree}
Let $\mathcal{T}$ a tangle of order $k$ in a connectivity system $(E,\lambda)$, and let $\mathcal{S}$ be a tree compatible set. Let $T$ be a partial $(k,\mathcal{S})$-tree for $\mathcal{T}$, and let $B$ be an $\mathcal{S}$-terminal-bag of $T$ labelling a leaf $w$ of $T$. If $X\subseteq E-B$ is a non-empty $\mathcal{T}$-weak set such that $B\cup X$ is $k$-separating, then there is a partial $(k,\mathcal{S})$-tree $T'$ that is equivalent to $T$ such that $B\cup X$ is an $\mathcal{S}$-terminal-bag of $T'$.
\end{lemma}

\begin{proof}
Suppose that $X\subseteq E-B$ is a non-empty $\mathcal{T}$-weak set such that $B\cup X$ is $k$-separating. If $\Phi=(P_1,\ldots, P_n)$ is a $k$-flower in $\mathcal{T}$ corresponding to a $k$-flower vertex of $T$, then we may assume that the petals of $\Phi$ are labelled such that $B\subseteq P_1$. Let $T'$ be the $\pi$-labelled tree obtained by relabelling $T$ such that:
\begin{enumerate}
 \item[(i)] the leaf $w$ of $T'$ is labelled by the bag $B\cup X$, and, if $u\neq w$ is a bag vertex of $T$ labelled by the bag $B'$, then $u$ is a bag vertex of $T'$ labelled by the bag $B'-X$; and
 \item[(ii)] if $v$ is a non-bag vertex of $T$ labelled by $X\in \{D,A\}$, then $v$ is a non-bag vertex of $T'$ labelled by $X$. Moreover, if a cyclic ordering is imposed on the edges of $T$ that are incident with $v$, then the cyclic ordering is imposed on the edges of $T'$ that are incident with $v$.
\end{enumerate}
It is clear that $T$ and $T'$ have the same bag vertices and non-bag vertices. We also see that $B\cup X$ is an $\mathcal{S}$-terminal-bag of $T '$ by (i), Lemma \ref{bigeqlemma} (ii) and (S1). It remains, then, to show that $T'$ is a partial $(k,\mathcal{S})$-tree for $\mathcal{T}$ that is equivalent to $T$. It follows immediately from (ii) that $T'$ satisfies (P2). In the following three sublemmas we show that $T'$ satisfies the remaining partial $(k,\mathcal{S})$-tree axioms.
\begin{claim}
\label{bigeqtree.1}
If $v$ is a non-bag vertex of $T'$, and $\Phi$ is the $k$-flower corresponding to the $k$-flower vertex $v$ of $T$, then the partition of $E$ displayed by the components of $T'-v$ is a $k$-flower $\Phi '$ that is equivalent to $\Phi$ and has no $\mathcal{T}$-loose petals.
\end{claim}

\begin{subproof}
Assume that the vertex $v$ of $T'$ is a non-bag vertex, and that $\Phi=(P_1, P_2, \ldots, P_n)$ is the $k$-flower corresponding to the $k$-flower vertex $v$ of $T$. Then the partition of $E$ displayed by $T'-v$, with the same ordering of the components as $T-v$, is $\Phi '=(P_1\cup X, P_2-X,\ldots, P_n-X)$. The set $P_1\cup X$ is the union of the $k$-separating sets $P_1$ and $B\cup X$, whose intersection is $B\cup (P_1\cap X)$. Since both $B\cup (P_1\cap X)$ and $E-(B\cup (P_1\cap X))$ are $\mathcal{T}$-strong, it follows from (T2) that $\lambda(B\cup (P_1\cap X))\geq k$. Thus, by uncrossing $P_1$ and $B\cup X$, the set $P_1\cup X$ is $k$-separating. It now follows from Lemma \ref{flowereq3} that $\Phi '$ is a $k$-flower that is equivalent to $\Phi$, and that $\Phi '$ has no $\mathcal{T}$-loose petals.
\end{subproof}

\begin{claim}
\label{bigeqtree.2}
 If $e$ is an edge of $T'$, then the partition of $E$ displayed by the components of $T'\del e$ is a $\mathcal{T}$-strong $k$-separation. Moreover, if $e$ is incident with two bag vertices of $T'$, then the partition of $E$ displayed by the components of $T'\del e$ is a $(k,\mathcal{S})$-separation that is $\mathcal{T}$-equivalent to the $(k,\mathcal{S})$-separation displayed by the components of $T\del e$.
\end{claim}

\begin{subproof}
Let $e$ be an edge of $T'$. If $e$ is incident with a $k$-flower vertex of $T'$, then the partition of $E$ displayed by the components of $T'\del e$ is a $\mathcal{T}$-strong $k$-separation by \ref{bigeqtree.1}. We may therefore assume that $e$ is incident with two bag vertices of $T'$. Then $e$ is also incident with two bag vertices of $T$, so the partition of $E$ displayed by the components of $T\del e$ is a $(k,\mathcal{S})$-separation $(R,G)$ by (P1). Now $(B,E-B)$ is also a $k$-separation displayed by an edge of $T$, so it follows from Lemma \ref{laminar} that $(B,E-B)$ does not cross $(R,G)$. Thus we may assume, up to switching $R$ and $G$, that $B\subseteq R$ because $B$ is a bag of $T$. Then $(R\cup X, G-X)$ is the partition of $E$ displayed by the components of $T'\del e$, and $R\cup X$ is $k$-separating by uncrossing $B\cup X$ and $R$, so $(R\cup X, G-X)$ is $\mathcal{T}$-equivalent to $(R,G)$ by Lemma \ref{bigeqlemma} (ii). Hence $(R\cup X, G-X)$ is a $(k,\mathcal{S})$-separation by (S1).
\end{subproof}
   
\begin{claim}
\label{bigeqtree.3}
 Every $(k,\mathcal{S})$-separation of $\lambda$ conforms with $T'$.
\end{claim}

\begin{subproof}
 Seeking a contradiction, suppose that $(R,G)$ is a $(k,\mathcal{S})$-separation that does not conform with $T'$. Then $(R,G)$ conforms with $T$ because $T$ is a partial $(k,\mathcal{S})$-tree, so, by possibly replacing $(R,G)$ by an equivalent $(k,\mathcal{S})$-separation, we may assume that either $(R,G)$ is displayed by $T$ or $R$ is contained in a bag of $T$. If $(R,G)$ is displayed by $T$, then it follows immediately from \ref{bigeqtree.1} and \ref{bigeqtree.2} that there is some $(k,\mathcal{S})$-separation that is equivalent to $(R,G)$ and displayed by $T'$. Thus we may assume that $R\subseteq B'$ for some bag $B'\not=B$ of $T$. We may further assume that both $R\cap (B'-X)$ and $R\cap X$ are non-empty, since $B'-X$ is a bag of $T'$ by (i). We now show that $(R-X, G\cup X)$ is a $(k,\mathcal{S})$-separation that is equivalent to $(R,G)$. Since $B\subseteq G$, the set $G\cup X$ is $k$-separating by uncrossing $B\cup X$ and $G$, so, by Lemma \ref{bigeqlemma}(ii) the $\mathcal{T}$-strong $k$-separation $(R-X,G\cup X)$ is equivalent to $(R,G)$. Hence $(R-X,G\cup X)$ is a $(k,\mathcal{S})$-separation by (S1). But $R-X\subseteq B'-X$, so $(R,G)$ conforms with $T'$; a contradiction. 
\end{subproof}
 
It follows from \ref{bigeqtree.1}, \ref{bigeqtree.2}, and \ref{bigeqtree.3} that $T'$ is a partial $(k,\mathcal{S})$-tree for $\mathcal{T}$. Moreover, it follows from  \ref{bigeqtree.1} and \ref{bigeqtree.2} that $T$ and $T'$ are equivalent partial $(k,\mathcal{S})$-trees.  
\end{proof}

\begin{cor}
\label{fclbag} 
Let $\mathcal{T}$ a tangle of order $k$ in a connectivity system $(E,\lambda)$, and let $\mathcal{S}$ be a tree compatible set. Let $T$ be a partial $(k,\mathcal{S})$-tree for $\mathcal{T}$, and let $B$ be an $\mathcal{S}$-terminal-bag of $T$. If $(X_i)_{i=1}^m$ is a partial $k$-sequence for $B$, then there is a partial $(k,\mathcal{S})$-tree $T'$ that is equivalent to $T$ such that $B\cup (\bigcup_{i=1}^m X_i)$ is an $\mathcal{S}$-terminal-bag of $T'$. 
\end{cor}

We next show that if $(X_i)_{i=1}^m$ is a partial $k$-sequence for $E-B$, then there is a partial $(k,\mathcal{S})$-tree that is $\mathcal{T}$-equivalent to $T$ with terminal bag $B-(\bigcup_{i=1}^m X_i)$.

\begin{lemma}
\label{treee}
Let $\mathcal{T}$ a tangle of order $k$ in a connectivity system $(E,\lambda)$, and let $\mathcal{S}$ be a tree compatible set. Let $T$ be a partial $(k,\mathcal{S})$-tree for $\mathcal{T}$, and let $B$ be an $\mathcal{S}$-terminal-bag of $T$. If $X\subseteq B$ is a non-empty $\mathcal{T}$-weak set such that $B-X$ is $k$-separating, then there is a partial $(k,\mathcal{S})$-tree $T'$ that is equivalent to $T$ such that $B-X$ is an $\mathcal{S}$-terminal-bag of $T'$.
\end{lemma}

\begin{proof}
Assume that $X\subseteq B$ is a non-empty $\mathcal{T}$-weak set such that $B-X$ is $k$-separating. Let $u$ be the bag vertex of $T$ that is labelled by $B$. We modify $T$ to produce a $\pi$-labelled tree $T'$ by adding a new vertex $v$ adjacent to $u$, relabelling the vertex $u$ by the bag $X$, and labelling $v$ by $B-X$. Then $B-X$ is an $\mathcal{S}$-terminal-bag of $T'$ by Lemma \ref{bigeqlemma} (ii) and (S1). It is easily verified that $T'$ satisfies the first four partial $(k,\mathcal{S})$-tree axioms, (P1)-(P4). Assume that $T'$ does not satisfy the axiom (P5). Then there is a $(k,\mathcal{S})$-separation $(R,G)$ that does not conform with $T'$. Since $T$ is a partial $(k,\mathcal{S})$-tree and $T'$ only differs from $T$ by adding $v$ and changing the bag $B$, we may assume, by possibly replacing $(R,G)$ by an equivalent $(k,\mathcal{S})$-separation, that $R\subseteq B$ and that both $R\cap X$ and $R\cap (B-X)$ are non-empty. Now, the set $G\cup X$ is the union of the $k$-separating sets $(E-B)\cup X$ and $G$. Since $((E-B)\cup X)\cap G$ contains $E-B$, and $E-(((E-B)\cup X)\cap G)$ contains $R$, the partition $(((E-B)\cup X)\cap G, E-(((E-B)\cup X)\cap G))$ is $\mathcal{T}$-strong, so $\lambda(((E-B)\cup X)\cap G)\geq k$ by (T2). Thus $G\cup X$ is $k$-separating by uncrossing $(E-B)\cup X$ and $G$, and so $(R-X, G\cup X)$ is equivalent to $(R,G)$ by Lemma \ref{bigeqlemma} (ii). Hence $(R-X, G\cup X)$ is a $(k,\mathcal{S})$-separation by (S1). But $R-X\subseteq B-X$, so $(R,G)$ conforms with $T'$; a contradiction. Thus $T'$ is indeed a partial $(k,\mathcal{S})$-tree. We now show that $T$ and $T'$ are equivalent partial $(k,\mathcal{S})$-trees. It is clear that $T\preccurlyeq_{\mathcal{S}} T'$. On the other hand, with the exception of $(B-X,E-(B-X))$, every $(k,\mathcal{S})$-separation displayed by $T'$ is also displayed by $T$. But $(B-X,E-(B-X))$ is equivalent to $(B,E-B)$ by Lemma \ref{bigeqlemma} (ii), and $(B,E-B)$ is displayed by $T$. Thus we also have $T'\preccurlyeq_{\mathcal{S}} T$, so $T$ and $T'$ are equivalent partial $(k,\mathcal{S})$-trees for $\mathcal{T}$. 
\end{proof}

\begin{cor}
 \label{treeecor}
Let $\mathcal{T}$ a tangle of order $k$ in a connectivity system $(E,\lambda)$, and let $\mathcal{S}$ be a tree compatible set. Let $T$ be a partial $(k,\mathcal{S})$-tree for $\mathcal{T}$, and let $B$ be an $\mathcal{S}$-terminal-bag of $T$. If $(X_i)_{i=1}^m$ is a partial $k$-sequence for $E-B$, then there is a partial $(k,\mathcal{S})$-tree $T'$ that is equivalent to $T$ with $\mathcal{S}$-terminal-bag $B-(\bigcup_{i=1}^m X_i)$.
\end{cor}

\begin{lemma}
\label{openpath}
Let $\mathcal{T}$ a tangle of order $k$ in a connectivity system $(E,\lambda)$, and let $\mathcal{S}$ be a tree compatible set. Let $T$ be a partial $(k,\mathcal{S})$-tree for $\mathcal{T}$, and let $B$ be an $\mathcal{S}$-terminal-bag of $T$. If $(C,E-C)$ is a $(k,\mathcal{S})$-separation such that $\fcl_{\mathcal{T}}(B)=\fcl_{\mathcal{T}}(C)$, then there is a partial $(k,\mathcal{S})$-tree $T '$ that is equivalent to $T$ with terminal bag $C$. 
\end{lemma}

\begin{proof}
Assume that $(C,E-C)$ is a $(k,\mathcal{S})$-separation such that $\fcl_{\mathcal{T}}(B)= \fcl_{\mathcal{T}}(C)$. Then $(C,E-C)$ is $\mathcal{T}$-equivalent to $(B,E-B)$ by Lemma \ref{3.3}. Let $(X_i)_{i=1}^n$ be a maximal partial $k$-sequence for $B$, and let $(Y_i)_{i=1}^m$ be a maximal partial $k$-sequence for $C$. Then $\fcl_{\mathcal{T}}(B)=B\cup (\bigcup_{i=1}^n X_i)$ and $\fcl_{\mathcal{T}}(C)=C\cup (\bigcup_{i=1}^m Y_i)$ by Lemma \ref{maxsequence}. By Corollary \ref{fclbag}, there is a partial $(k,\mathcal{S})$-tree $T'$ for $\mathcal{T}$ that is equivalent to $T$ such that $\fcl_{\mathcal{T}}(B)$ is an $\mathcal{S}$-terminal-bag of $T '$. Now $(Y_{m-i+1})_{i=1}^m$ is a partial $k$-sequence for $E-\fcl_{\mathcal{T}}(B)$, so, by Corollary \ref{treeecor}, there is a partial $(k,\mathcal{S})$-tree $T ''$ for $\mathcal{T}$ that is equivalent to $T '$, and hence equivalent to $T$, with terminal bag $C$, as required. 
\end{proof}

\section{Proof of the main theorem}
\label{mainone}

We can now prove Lemma \ref{tree}, from which Theorem \ref{bigone} will easily follow.

\begin{proof}[Proof of Lemma \ref{tree}.]
Suppose that $(R,G)$ is a $(k,\mathcal{S})$-separation that is not equivalent to any $(k,\mathcal{S})$-separation displayed by $T$. Then $(R,G)$ conforms with $T$ by (P5), so we may assume, by possibly replacing $(R,G)$ by a $\mathcal{T}$-equivalent $(k,\mathcal{S})$-separation, that $R$ is properly contained in a bag $B$ of $T$. Let $u$ be the vertex of $T$ labelled by $B$. We distinguish two cases:
\begin{enumerate}
 \item[(I)] $u$ is a leaf of $T$; and
 \item[(II)] $u$ is not a leaf of $T$.
\end{enumerate}
Consider case (I). 

\begin{claim}
 $(B,E-B)$ is a $(k,\mathcal{S})$-separation.
\end{claim}

\begin{subproof}
 If $u$ is adjacent to a bag vertex, then the result follows immediately from (P1). Assume that $u$ is adjacent to a $k$-flower vertex $v$, and let $\Phi=(P_1,\ldots, P_n)$ be the $k$-flower corresponding to $v$. Then $B$ contains $R$, so $B\in \mathcal{S}$ by (S2). On the other hand, $B$ is contained in a petal of $\Phi$, so $E-B$ contains $E-P_i$ for some petal $P_i$ of $\Phi$. Then $E-B\in \mathcal{S}$ by Lemma \ref{nompetalt} and (S2). Thus $(B,E-B)$ is a $(k,\mathcal{S})$-separation.
\end{subproof}

Now by Corollary \ref{treeecor} we may assume, by possibly replacing $T$ by an equivalent partial $(k,\mathcal{S})$-tree and replacing $(R,G)$ by an equivalent $(k,\mathcal{S})$-separation, that $E-B$ is fully closed with respect to $\mathcal{T}$. Let $Z$ be a $k$-separating set that is maximal with respect to the property that $R\subseteq Z\subsetneq B$. Let $(W,Z)=(E-Z,Z)$.

\begin{claim}
\label{blah}
 $(W,Z)$ is a $(k,\mathcal{S})$-separation that is not equivalent to any $(k,\mathcal{S})$-separation that is displayed by $T$.
\end{claim}

\begin{subproof}
 Since $E-B$ is contained in $W$, and $R$ is contained in $Z$, it follows from (S2) that $(W,Z)$ is a $(k,\mathcal{S})$-separation. Now, seeking a contradiction, suppose that $(W,Z)$ is equivalent to a $(k,\mathcal{S})$-separation $(W',Z')$ that is displayed by $T$, with labels chosen such that $\fcl_{\mathcal{T}}(Z)= \fcl_{\mathcal{T}}(Z')$. Since $(W',Z')$ is non-sequential, and $\fcl_{\mathcal{T}}(Z)= \fcl_{\mathcal{T}}(Z')$, it follows that $Z'$ meets $Z$, and so $Z'$ meets $B$. But $Z'$ is a union of bags of $T$, so $B$ is contained in $Z'$. Thus $\fcl_{\mathcal{T}}(Z)=\fcl_{\mathcal{T}}(B)$ by Lemma \ref{fcloperator}, so we also have $\fcl_{\mathcal{T}}(W) =\fcl_{\mathcal{T}}(E-B)$ by Lemma \ref{3.3}. But $E-B$ is fully closed, so it follows that $W\subseteq E-B$; a contradiction because $E-B\subsetneq W$ by the choice of $Z$.
\end{subproof}

We note that, by \ref{blah}, the set $B\cap W$ is $\mathcal{T}$-strong, so the partition $(B\cap W, E-(B\cap W))$ is $\mathcal{T}$-strong. Hence $\lambda(B\cap W)\geq k$ by (T2). 

\begin{claim}
\label{win3}
 If $B\cap W$ is not $k$-separating, then there is a partial $(k,\mathcal{S})$-tree $T'$ such that $T'\succcurlyeq_{\mathcal{S}} T$ and $T'$ displays $(W,Z)$.
\end{claim}

\begin{subproof}
  Assume that $B\cap W$ is not $k$-separating. Let $T'$ be the tree that is obtained from $T$ by adjoining a new leaf $v$ adjacent to $u$ such that $v$ is a bag vertex labelled by $Z$, and $u$ is relabelled by $B\cap W$. It is easily verified that $T'$ satisfies the first four partial $(k,\mathcal{S})$-tree axioms, (P1)-(P4). Assume that it does not satisfy (P5). Then there is a $(k,\mathcal{S})$-separation $(X,Y)$ that does not conform with $T'$. Since $(X,Y)$ conforms with the partial $(k,\mathcal{S})$-tree $T$, and $T'$ only differs from $T$ by adding $v$ and changing the bag $B$, we may assume, by possibly replacing $(X,Y)$ by an equivalent $(k,\mathcal{S})$-separation, that $X\subsetneq B$ and that both $X\cap Z$ and $B\cap W\cap X$ are non-empty. Assume first that $\lambda(X\cap Z)<k$. Since $E-(X\cap Z)$ is $\mathcal{T}$-strong, it follows that $X\cap Z$ is a member of $\mathcal{T}$ by (T2). Then the partition $(Z-X, E-(Z-X))$ is $\mathcal{T}$-strong, since $(Z,E-Z)$ is non-sequential, so $\lambda(Z-X)\geq k$ by (T2). Now by uncrossing $Y$ and $Z$, whose intersection is $Z-X$, we see that $Y\cup Z$ is $k$-separating. Thus $(X,Y)$ is equivalent to $(X-Z, Y\cup Z)$ by Lemma \ref{bigeqlemma} (ii). But $(X-Z, Y\cup Z)$ conforms with $T'$; a contradiction. Thus we may now assume that $\lambda(X\cap Z)\geq k$. Then $X\cup Z$ is $k$-separating by uncrossing $X$ and $Z$. If $X\cup Z$ is properly contained in $B$, then $X\cup Z$ contradicts our choice of $Z$. Thus we may assume that $X\cup Z=B$, and hence that $B\cap W=W\cap X$. Then $\lambda(W\cap X)=\lambda(B\cap W)>k$ because $B\cap W$ is not $k$-separating. Thus $\lambda(W\cup X)\leq \lambda(W)+\lambda(X)-\lambda(W\cap X)<k$ by the submodularity of $\lambda$, and $W\cup X$ is $\mathcal{T}$-strong, so its complement $Z-X$ is a member of $\mathcal{T}$ by (T2). It now follows from Lemma \ref{bigeqlemma} (ii) that $(X,Y)$ is equivalent to $(B,E-B)$. But $(B,E-B)$ is displayed by $T'$, so $(X,Y)$ conforms with $T'$; a contradiction. It follows from this contradiction that $T'$ is indeed a partial $(k,\mathcal{S})$-tree. Clearly $T'\succcurlyeq_{\mathcal{S}} T$. Moreover, the $(k, \mathcal{S})$-separation $(W,Z)$ is displayed $T'$ but not $T$. 
\end{subproof}
Thus, by \ref{win3}, we may now assume that $B\cap W$ is $k$-separating. Then $\Phi=(Z, B\cap W, E-B)$ is a $k$-flower in $\mathcal{T}$. Let $\Phi '=(P_1,P_2,\ldots,P_n)$ be an $\mathcal{S}$-tight $\mathcal{S}$-maximal $k$-flower in $\mathcal{T}$ such that $\Phi '\succcurlyeq_{\mathcal{S}} \Phi$. Then $\Phi '$ displays a $(k,\mathcal{S})$-separation $(C,E-C)$ that is equivalent to $(B,E-B)$. Thus we may assume that $\fcl_{\mathcal{T}}(B)= \fcl_{\mathcal{T}}(C)$. We observe that, since $E-B$ is fully closed, the set $B$ is contained in $C$. Hence $Z$ is contained in $C$. We may also assume, up to labels, that $C=P_1\cup \cdots \cup P_j$ for some $j\in [n-1]$. Now $\Phi '$ also displays a $(k,\mathcal{S})$-separation $(W',Z')$ that is equivalent to $(W,Z)$. Since $(C,E-C)$ and $(W',Z')$ are inequivalent $(k,\mathcal{S})$-separations, and $Z\subseteq C$, we may assume, by Lemma \ref{confconcatcond}, that $Z'\subseteq C$. Thus, both $(C,E-C)$ and $(W',Z')$ are displayed by the concatenation $\Phi ''=(P_1,\ldots, P_j, P_{j+1}\cup \cdots \cup P_n)$ of $\Phi$.

By Lemma \ref{openpath} there is a partial $(k,\mathcal{S})$-tree $T'$ that is equivalent to $T$ with terminal bag $C$ labelling the vertex $u$. We now let $T''$ be the $\pi$-labelled tree that is obtained from $T'$ as follows: we first adjoin a new flower vertex $v$ adjacent to $u$; then adjoin bag vertices $v_1,\ldots, v_j$ adjacent to $v$ labelling these by $P_1,\ldots,P_j$ respectively; label $v$ by $D$ or $A$ according to the type of $\Phi ''$, and, if necessary, we impose the cyclic order $(vv_1,\ldots, vv_j,vw)$ on the edges incident with $v$; finally, we relabel $u$ by $\emptyset$. We claim that $T''$ is a partial $(k,\mathcal{S})$-tree such that $T''\succcurlyeq_{\mathcal{S}} T$ and that $T''$ displays a $(k,\mathcal{S})$-separation that is equivalent to $(W,Z)$.
 
It is easily verified that (P1) and (P2) hold for $T''$. It is also clear that the partition of $E$ displayed by $T''-v$ is the $k$-flower $\Phi ''=(P_1, \ldots, P_j, P_{j+1}\cup \cdots \cup P_n)$, and we have seen that $\Phi ''$ displays at least two inequivalent $(k,\mathcal{S})$-separations, so $\Phi ''$ has $\mathcal{S}$-order at least 3. Moreover, it follows from Lemma \ref{tightconcatcond} that $\Phi ''$ has no $\mathcal{T}$-loose petals. Thus it follows that the axioms (P3) and (P4) hold for $T''$. Assume that $T''$ does not satisfy (P5). Then there is some $(k,\mathcal{S})$-separation $(X,Y)$ that does not conform with $T ''$. Since $(X,Y)$ conforms with $T'$, and $T''$ only differs from $T'$ by changing the bag $C$, it follows that, by possibly replacing $(X,Y)$ by an equivalent $(k,\mathcal{S})$-separation, we may assume that $X\subseteq C$, and that $X$ is not contained in the bag $P_i$ of $T''$ for any $i\in [j]$. Because $\Phi'$ is an $\mathcal{S}$-tight $\mathcal{S}$-maximal flower in the robust tangle $\mathcal{T}$, it follows from Theorem \ref{TMkflowerconf} that $(X,Y)$ conforms with $\Phi '$. Thus there is a $(k,\mathcal{S})$-separation $(X',Y')$ that is equivalent to $(X,Y)$ with $\fcl_{\mathcal{T}}(X)=\fcl_{\mathcal{T}}(X')$ such that either:
\begin{enumerate}
 \item[(i)] $(X',Y')$ is displayed by $\Phi'$; or
 \item[(ii)] $X'$ or $Y'$ is contained in a petal of $\Phi'$.
\end{enumerate}
Assume first that (i) holds. Then, by Lemma \ref{confconcatcond} and since $X\subseteq C$, we may assume that $X'\subseteq C$. Then $(X',Y')$ is displayed by $\Phi''$, and hence is displayed by $T''$; a contradiction. Thus we may assume that (ii) holds. Suppose that $X'$ is contained in a petal of $\Phi'$. Then, since $\fcl_{\mathcal{T}}(X)=\fcl_{\mathcal{T}}(X')$, the set $X\cap X'$ is non-empty, and $X\subseteq C$, so $X'\subseteq P_i$ for some $i\in [j]$. Hence $X'$ is contained in a bag of $T''$; a contradiction. Assume then that $Y'$ is contained in a petal $P_i$ of $\Phi$. If $i\in [j]$, then $Y'$ is contained in a bag of $T''$. Assume that $i\in [j+1,n]$. Then $X\subseteq C\subseteq X'$, so $(C,E-C)$ is a $(k,\mathcal{S})$-separation that is equivalent to $(X',Y')$ by Lemma \ref{fcloperator}, Lemma \ref{3.3} and (S1). But $(C,E-C)$ is displayed by $T''$, so $(X,Y)$ conforms with $T''$; a contradiction. Thus every $(k,\mathcal{S})$-separation conforms with $T''$. Therefore $T''$ is indeed a partial $(k,\mathcal{S})$-tree. Clearly $T''\succcurlyeq_{\mathcal{S}} T'$, so $T''\succcurlyeq_{\mathcal{S}} T$. Moreover, $T''$ displays a $(k,\mathcal{S})$-separation that is equivalent to $(W,Z)$. Therefore the lemma holds for case (I).

Consider case (II). Choose a $k$-separating set $Z$ that is maximal with respect to the property that $R\subseteq Z\subseteq B$. Then $E-Z$ also contains a member of $\mathcal{S}$, since $T$ is non-trivial, so $(Z,E-Z)$ is a $(k,\mathcal{S})$-separation by (S2). Let $T'$ be the $\pi$-labelled tree obtained from $T$ by adjoining a new leaf $v$ adjacent to $u$ such that $v$ is a bag vertex labelled by $Z$, and $u$ is relabelled by $B-Z$. 

\begin{claim}
 $T'$ is a partial $(k,\mathcal{S})$-tree for $\mathcal{T}$, and $T\preccurlyeq_{\mathcal{S}} T'$.
\end{claim}

\begin{subproof}
  The first four partial $(k,\mathcal{S})$-tree axioms, (P1)-(P4), hold immediately for $T'$. Assume that (P5) does not hold for $T'$. Then there is a $(k,\mathcal{S})$-separation $(Y,E-Y)$ that does not conform with $T'$. Since $(Y,E-Y)$ conforms with $T$, and $T'$ only differs from $T$ by adding $v$ and changing the bag $B$, we may assume, by possibly replacing $(Y,E-Y)$ by an equivalent $(k,\mathcal{S})$-separation, that $Y\subseteq B$ and that both $Y\cap Z$ and $Y\cap (B-Z)$ are non-empty. Assume that $Y\cap Z$ is $\mathcal{T}$-weak. It follows that $\lambda(Z-Y)\geq k$, since the $(k,\mathcal{S})$-separation $(Z,E-Z)$ is non-sequential. But $Z-Y=E-(Y\cup (E-Z))$, so $Y\cap (E-Z)=Y-Z$ is $k$-separating by uncrossing the $k$-separating sets $Y$ and $E-Z$. Then, by Lemma \ref{bigeqlemma} (ii) and (S1), the $k$-separation $(Y-Z, E-(Y-Z))$ is a $(k,\mathcal{S})$-separation that is equivalent to $(Y,E-Y)$. But $(Y-Z, E-(Y-Z))$ conforms with $T'$; a contradiction. Thus we may assume that $Y\cap Z$ is $\mathcal{T}$-strong. Then the partition $(Y\cap Z, E-(Y\cap Z))$ is $\mathcal{T}$-strong, so $\lambda(Y\cap Z)\geq k$ by (T2). Then $Y\cup Z$ is $k$-separating by uncrossing $Y$ and $Z$; a contradiction of the maximality of $Z$. Thus $T'$ satisfies (P5), so $T'$ is a indeed a partial $(k,\mathcal{S})$-tree, and moreover $T\preccurlyeq_{\mathcal{S}} T'$.
\end{subproof}

Now $Z$ labels a leaf of $T'$, and $R\subseteq Z$, so it follows by case (I) that there is a partial $(k,\mathcal{S})$-tree $T''\succcurlyeq_{\mathcal{S}} T'$ such that $T''$ displays a $(k,\mathcal{S})$-separation that is not equivalent to any $(k,\mathcal{S})$-separation displayed by $T'$. Hence $T''\succcurlyeq_{\mathcal{S}} T$ and $T''$ displays a $(k,\mathcal{S})$-separation that is not $\mathcal{T}$-equivalent to any $(k,\mathcal{S})$-separation displayed by $T$.
\end{proof}

At last we can prove our main theorem.

\begin{thm}
\label{bigone}
Let $\mathcal{T}$ be a robust tangle of order $k$ in a connectivity system $(E,\lambda)$, and let $\mathcal{S}$ be a tree compatible set. If $T$ is an $\mathcal{S}$-maximal partial $(k,\mathcal{S})$-tree for $\mathcal{T}$, then every $(k,\mathcal{S})$-separation of $\lambda$ is $\mathcal{T}$-equivalent to some $(k,\mathcal{S})$-separation displayed by $T$.
\end{thm}

\begin{proof}
Assume that $T$ is an $\mathcal{S}$-maximal partial $(k,\mathcal{S})$-tree for $\mathcal{T}$. If there are no $(k,\mathcal{S})$-separations of $\lambda$, then the theorem holds. Suppose that $(R,G)$ is a $(k,\mathcal{S})$-separation of $\lambda$. Then, by Lemma \ref{TMflow}, there is an $\mathcal{S}$-tight $\mathcal{S}$-maximal $k$-flower in $\mathcal{T}$ that displays a $(k,\mathcal{S})$-separation equivalent to $(R,G)$, and so, by Corollary \ref{TMkflower}, there is a partial $(k,\mathcal{S})$-tree $T'$ for $\mathcal{T}$ that displays a $(k,\mathcal{S})$-separation equivalent to $(R,G)$. Thus we may assume that $T$ is a non-trivial partial $(k,\mathcal{S})$-tree for $\mathcal{T}$. Then the theorem holds, or else, by Lemma \ref{tree}, we contradict the $\mathcal{S}$-maximality of $T$.
\end{proof}

If $\mathcal S$ is the collection of $k$-separations that are non-sequential with respect
to $\tangle$, then we will call an $\mathcal S$-maximal partial $(k,{\mathcal S})$-tree
for $\tangle$ a {\em maximal partial $k$-tree} for $\tangle$. For such
$k$-separations Theorem~\ref{bigone} becomes 

\begin{cor}
\label{bigone2}
Let $\tangle$ be a robust tangle of order $k$ in a connectivity system 
$(E,\lambda)$. If $T$ is a maximal partial $k$-tree for 
$\tangle$, then every $k$-separation of $\lambda$ that is non-sequential with
respect to $\tangle$ is equivalent to a $k$-separation displayed by $T$.
\end{cor}

\subsection*{Vertically $k$-connected matroids} We now interpret Corollary~\ref{bigone2} for vertically $k$-connected
matroids. Recall that if $M$ is a vertically $k$-connected matroid whose rank is
at least $\max\{3k-5,2\}$, then $M$ has a unique tangle $\tangle_k$ of order $k$.

\begin{cor}
\label{bigone3}
Let $M$ be a vertically $k$-connected matroid where $k\geq 2$ and 
$r(M)\geq \max\{8k-15,2\}$. Let $T$ be a maximal partial $k$-tree
for $\tangle_k$. Then every $k$-separation of $M$ that is non-sequential
with respect to $\tangle_k$ is equivalent to some $k$-separation displayed by $T$.
\end{cor}

Finally we note that, if $M$ is strictly $k$-connected, then, in Corollary~\ref{bigone3}, 
we may replace
the condition that $r(M)\geq \max\{8k-15,2\}$ by the condition that
$|E(M)|\geq \max\{8k-15,2\}$.

\bibliographystyle{acm}
\bibliography{tangletree2}
\end{document}